\documentclass[12pt]{article}
\usepackage{geometry} 
\usepackage{amssymb}
\usepackage{amsmath}
\usepackage{mathrsfs}
\usepackage{amsthm}
\usepackage{stmaryrd}
\usepackage{bbm}

\theoremstyle{plain}
\newtheorem{theorem}{Theorem}[section]
\newtheorem{lemma}[theorem]{Lemma}

\newtheorem{corollary}[theorem]{Corollary}
\newtheorem{proposition}[theorem]{Proposition}

\theoremstyle{remark}

\newtheorem{definition}[theorem]{Definition}

\newtheorem{example}[theorem]{Example}

\newcommand{\tripnorm}{\vert\kern-0.25ex\vert\kern-0.25ex\vert}
\newcommand{\Bigtripnorm}{\Big\vert\kern-0.2ex\Big\vert\kern-0.2ex\Big\vert}

\newcommand{\nat}{\mathbb N}
\newcommand{\real}{\mathbb R}
\newcommand{\complex}{\mathbb C}
\newcommand{\asplundop}{\ensuremath{\mathscr{D}}}

\newcommand{\sepop}{\ensuremath{\mathscr{X}}}

\newcommand{\allop}{\ensuremath{\mathscr{L}}}

\newcommand{\opideal}{\ensuremath{\mathscr{I}}}

\providecommand{\spn}{\mathop{\rm span}}

\providecommand{\MIN}{\mathop{\rm MIN}}

\newcommand{\cee}{\mathcal{C}}
\newcommand{\dee}{\mathcal{D}}

\newcommand{\eff}{\mathcal{F}}

\newcommand{\ess}{\mathcal{S}}
\newcommand{\tee}{\mathcal{T}}
\newcommand{\yoo}{\mathcal{U}}
\newcommand{\veee}{\mathcal{V}}
\newcommand{\dubyoo}{\mathcal{W}}

\numberwithin{theorem}{section}
\numberwithin{equation}{section}

\begin{document}
\title{Non-Asplund Banach spaces and operators}
\author{Philip A.H. Brooker}

\maketitle

\begin{abstract} Let $W$ and $Z$ be Banach spaces such that $Z$ is separable and let $R:W\longrightarrow Z$ be a (continuous, linear) operator. We study consequences of the adjoint operator $R^\ast$ having non-separable range. From our main technical result we obtain applications to the theory of basic sequences and the existence of universal operators for various classes of operators between Banach spaces. We also obtain an operator-theoretic characterisation of separable Banach spaces with non-separable dual.
\end{abstract}


\section{Introduction}\label{squintro}
The 1970s saw significant progress made on the understanding of the Radon-Nikod{\'y}m property in Banach spaces. Amongst the main achievements in this area of investigation was the proof that a dual Banach space $X^\ast$ has the Radon-Nikod{\'y}m property if and only if every separable subspace of $X$ has separable dual. We refer the reader to the book \cite{Diestel1977} for references and a proof of this result, but mention here in particular that it was C. Stegall who in \cite{Stegall1975} provided the proof of the fact that if $X^\ast$ has the Radon-Nikod{\'y}m property then every separable subspace of $X$ has separable dual. The main technical result in Stegall's paper, Corollary~1 of \cite{Stegall1975}, provides structural consequences for a separable Banach space $X$ and its dual $X^\ast$ in the case that $X^\ast$ is non-separable. The main technical result of the current paper, Theorem~\ref{scoopie} below, serves a similar purpose in the study of separable Banach spaces with non-separable dual. Our approach, which in some ways is fundamentally similar to Stegall's, uses techniques developed recently in \cite{PAHBSzlenkLarge} to study the Szlenk index. We obtain new proofs of old results and a number of new results concerning Banach spaces and operators with non-separable dual.

One of the primary applications of Theorem~\ref{scoopie} of the current paper is to the theory of basic sequences in Banach spaces. Our work to this end is based on the classical method of Mazur for producing subspaces with a basis and the more recent method of Johnson and Rosenthal \cite{Johnson1972} for producing quotients with a basis. 

The other main application of our techniques is to the problem of finding universal elements for certain subclasses of the class $\allop$ of all (bounded, linear) operators between Banach spaces. For operators $T\in\allop(X,Y)$ and $S\in\allop(W,Z)$, where $W,X,Y$ and $Z$ are Banach spaces, we say that $S$ {\em factors through} $T$ (or, equivalently, that $T$ {\em factors} $S$) if there exist $U\in \allop(W,X)$ and $V\in\allop(Y,Z)$ such that $VTU=S$. With this terminology, for a given subclass $\mathscr{C}$ of $\allop$ we say that an operator $\Upsilon\in\mathscr{C}$ is {\em universal for} $\mathscr{C}$ if $\Upsilon$ factors through every element of $\mathscr{C}$. Typically $\mathscr{C}$ will be the complement $\complement\opideal$ of an operator ideal $\opideal$ in the sense of Pietsch \cite{Pietsch1980} (that is, $\complement\opideal$ consists of all elements of $\allop$ that do not belong to $\opideal$), or perhaps the restriction $\mathscr{J}\cap \complement\opideal$ of $\complement\opideal$ to a large subclass $\mathscr{J}$ of $\allop$; e.g., $\mathscr{J}$ might denote a large operator ideal or the class of all operators having a specified domain or codomain. One may think of a universal element of the class $\mathscr{C}$ as a minimal element of $\mathscr{C}$ that is `fixed' or `preserved' by each element of $\mathscr{C}$. 

The notion of universality for a class of operators goes back to the work of Lindenstrauss and {Pe{\l}czy{\'n}ski}, who showed in \cite{Lindenstrauss1968} that the summation operator $\Sigma: (a_n)_{n=1}^\infty \mapsto (\sum_{i=1}^n a_i)_{n=1}^\infty$ from $\ell_1$ to $\ell_\infty$ is universal for the class of non-weakly compact operators. Soon after, W.B.~Johnson \cite{Johnson1971a} showed that the formal identity operator from $\ell_1$ to $\ell_\infty$ is universal for the class of non-compact operators. This result of Johnson has been applied in the study of information-based complexity by Hinrichs, Novak and Wo{\'z}niakowski in \cite{Hinrichs2013}.

The universality result of primary importance for us is due to Stegall and establishes the existence of a universal non-Asplund operator. The Asplund operators have several equivalent definitions in the literature; in the current paper we say that an operator $T:X\longrightarrow Y$ is {\it Asplund} if $T\vert_Z\in\sepop^\ast$ for any separable subspace $Z\subseteq X$, where $\sepop^\ast$ denotes the class of operators whose adjoint has separable range. We refer the reader to Stegall's paper \cite{Stegall1981} for further properties and characterisations of Asplund operators.

Stegall's universal non-Asplund operator is defined in terms of the Haar system $(h_m)_{m=0}^\infty\subseteq C(\{ 0,1\}^\omega)$, where each factor $\{ 0,1\}$ is discrete and $\{0,1\}^\omega$ is equipped with its compact Hausdorff product topology. The Haar system is a monotone basis for $C(\{0,1\}^\omega)$ and may be defined as follows. Let $\dee$ denote the set $\bigcup_{n<\omega}\{0,1\}^n$, the set of finite sequences in $\{0,1\}$, and let $(t_m)_{m=1}^\infty$ be the enumeration of $\dee$ such that the following two conditions hold for all $n<\omega$: 
\begin{itemize}
\item[(i)] $\{0,1\}^n=\{t_m\mid 2^n\leq m<2^{n+1} \}$; and
\item[(ii)] For $2^n\leq l<m<2^{n+1}$ there exists $k<n$ such that $t_l(j)=t_m(j)$, $0\leq j<k$, and $t_l(k)=0$ and $t_m(k)=1$.
\end{itemize}
For each $t\in\dee$ let $\Delta_t$ denote the set of all elements of $\{0,1\}^\omega$ having $t$ as an initial segment (that is, for $t\in\{0,1\}^n$ and $s\in\dee$ we have $s\in\Delta_t$ if and only if  $s(k)=t(k)$ for all $k<n$). For $t\in\{0,1\}^n$ and $i\in\{0,1\}$ we let $t^\smallfrown i$ denote the element of $\{ 0,1\}^{n+1}$ satisfying $(t^\smallfrown i)(k)=t(k)$ for $k<n$ and $(t^\smallfrown i)(n)=i$. The Haar system is defined by setting $h_0=\chi_{\{0,1\}^\omega}$ and $h_m=\chi_{\Delta_{{t_m}^\smallfrown 0}}-\chi_{\Delta_{{t_m}^\smallfrown 1}}$ for $m\geq 1$. Let $\mu$ denote the product measure on $\{0,1\}^\omega$ obtained by equipping each factor $\{ 0,1\}$ with its discrete uniform probability measure and let $H:\ell_1\longrightarrow L_\infty(\{0,1\}^\omega,\mu)$ be defined by setting $Hx=\sum_{m=1}^\infty x(m)h_{m-1}$ for each $x=(x(m))_{m=1}^\infty\in \ell_1$. The following result is due to Stegall.

\begin{theorem}[Theorem 4 of \cite{Stegall1975}]\label{thewox}
Let $X$ and $Y$ be Banach spaces such that $X$ is separable and suppose $T:X\longrightarrow Y$ is an operator such that $T^\ast(Y^\ast)$ is nonseparable. Then there exist operators $U:\ell_1\longrightarrow X$ and $V:Y\longrightarrow L_\infty(\{0,1\}^\omega,\mu)$ such that $VTU=H$.
\end{theorem}
Since the domain of $H$, namely $\ell_1$, is separable, and since $H^\ast$ has non-separable range, $H$ is a non-Asplund operator. It therefore follows from Theorem~\ref{thewox} that $H$ is a universal non-Asplund operator.

Other universality results besides those mentioned above can be found in the work of Brooker \cite{PAHBSzlenkLarge}, Cilia and Guti{\'e}rrez \cite{Cilia2015}, Dilworth \cite{Dilworth1985}, Girardi and Johnson \cite{Girardi1997}, Hinrichs and Pietsch \cite{Hinrichs2000}, Oikhberg \cite{Oikhberg2016}, and the Handbook survey on operator ideals by Diestel, Jarchow and Pietsch \cite{Diestel2001}. Recent work of R.~Causey on the existence of non-Asplund spaces of type 2, in a similar vein to work of Pisier and Xu \cite{Pisier1987}, has produced another construction of a universal non-Asplund operator using interpolation techniques \cite{Causey2016}.

We note that there are various other senses in which a large operator may preserve some kind of mathematical structure; several examples related to the subject matter of the current paper are as follows. Rosenthal \cite{Rosenthal1972} has shown that for a Banach space $Y$ and operator $T: C(\{ 0,1\}^\omega)\longrightarrow Y$, the adjoint operator $T^\ast$ has non-separable range if and only if $T$ fixes an isomorphic copy of $C(\{ 0,1\}^\omega)$. For separable Banach spaces $X$ and $Y$ and an operator $T:X\longrightarrow Y$, Dodos has characterised non-separability of $T^\ast (Y^\ast)$ in terms of $T$ fixing an isomorphic copy of $\ell_1$ and/or a `topological copy' of the standard basis of the James tree space (Theorem 3 and Theorem 4 of \cite{Dodos2011}). Godefroy and Finet \cite{Finet1989} have shown that if $X$ is separable and does not contain a copy of $\ell_1$, then an operator $T:X^\ast\longrightarrow X^\ast$ has non-separable range if and only if $T$ `fixes' a biorthogonal system of cardinality that of the continuum.

We now outline the content of subsequent sections of the current paper. In Section~\ref{backgrounder} we establish notation used throughout the paper and provide the requisite background on trees. In Section~\ref{nonsnake} we prove our key result, Theorem~\ref{scoopie}, which is in a similar spirit to the following result of Stegall:
\begin{theorem}[Corollary~1 of \cite{Stegall1975}]\label{doodah}
Let $X$ be a separable Banach space, $K$ a subset of $X^\ast$ that is non-separable in the norm topology and $G_\delta$ in the weak$^\ast$ topology, and $\epsilon>0$. Then there exists a weak$^\ast$-homeomorphism $\Xi$ of $\{ 0,1\}^\omega$ onto a subset of $K$ and a bounded sequence $(x_m)_{m=1}^\infty$ in $X$ such that
\[
\sum_{m=1}^\infty \Vert Tx_m-\chi_{\Delta_{t_m}}\Vert \leq\epsilon,
\]
where $T: X\longrightarrow C(\{0,1\}^\omega)$ is the operator satisfying $(Tx)(s)=\langle \Xi(s),x\rangle$ for every $x\in X$ and $s\in\{0,1\}^\omega$.
\end{theorem}The sequence $(t_m)_{m=1}^\infty$ in the statement of Theorem~\ref{doodah} is as defined in the paragraph preceding Theorem~\ref{thewox}. Theorem~\ref{doodah} is used by Stegall in \cite{Stegall1975} to derive several other substantial results for non-Asplund spaces and operators, amongst them Theorem~\ref{thewox} above. The key difference between Theorem~\ref{doodah} and Theorem~\ref{scoopie} is essentially that, under the stronger hypothesis that $K$ is absolutely convex - as in the statement of Theorem~\ref{scoopie} - we may deduce from Theorem~\ref{scoopie} that the statement of Theorem~\ref{doodah} holds with $\epsilon=0$. 

In Section~\ref{basissection} we apply Theorem~\ref{scoopie} to give a new proof of the following known result: if $X$ is a separable Banach space with $X^\ast$ non-separable, then $X$ admits a subspace $Y$ with a basis and with $Y^\ast$ non-separable (Theorem~\ref{secondbasistheorem}). For such $X$, Theorem~\ref{scoopie} also yields the existence of a subspace $Z$ of $X$ such that $X/Z$ has a basis and $(X/Z)^\ast$ is non-separable (Theorem~\ref{yetanotherbasistheorem}). In Section~\ref{universalsection} we turn our attention to universality results for certain classes of non-Asplund operators. We obtain Theorem~\ref{thewox} as a corollary of our main result Theorem~\ref{scoopie} and obtain stronger universality results for non-Asplund operators having separable codomain. In Section~\ref{pelchsection} we establish an operator-theoretic characterisation of separable Banach spaces with non-separable dual (Theorem~\ref{weakerbeaker}), analogous to Pe{\l}czy{\'n}ski's well-known operator-theoretic characterisation of separable Banach spaces containing a copy of $\ell_1$.

\section{Notation and preliminaries}\label{backgrounder}

We work with Banach spaces over either $\real$ or $\complex$. Typical Banach spaces are denoted by the letters $W$, $X$, $Y$ and $Z$, with the identity operator of $X$ denoted $Id_X$. We write $X^\ast$ for the dual space of $X$ and denote by $\imath_X$ the canonical embedding of $X$ into $X^{\ast\ast}$. We define $B_X:=\{ x\in X \mid\Vert x\Vert\leq 1\}$ and $B_X^\circ:= \{ x\in X \mid\Vert x\Vert< 1\}$. By a {\it subspace} of a Banach space $X$ we mean a vector subspace of $X$ that is closed in the norm topology. For a Banach space $X$ and sets $C\subseteq X$ and $D\subseteq X^\ast$ we define $C^\perp:= \{x^\ast \in X^\ast\mid \forall x\in C, \, x^\ast (x) =0 \}$ and $D_\perp = \{ x\in X\mid \forall x^\ast\in D, \, x^\ast (x) =0 \}$. We denote by $[C]$ the norm closed linear hull of $C$ in $X$, with a typical variation on this notation being that for an indexed set $\{ x_i\mid i\in I \}\subseteq X$ we may write $[x_i]_{i\in I}$ or $[x_i\mid i\in I]$ in place of $[\{x_i\mid i\in I\}]$. We shall make use of the well-known fact that, for a Banach space $X$ and a sequence $(x_m^\ast)_{m=1}^\infty\subseteq X^\ast$, the quotient map $Q: X\longrightarrow X/\bigcap_{m=1}^\infty\ker(x_m^\ast)$ has the property that $Q^\ast$ is an isometric weak$^\ast$-isomorphism of $(X/\bigcap_{m=1}^\infty\ker(x_m^\ast))^\ast$ onto the weak$^\ast$-closed linear hull of $(x_m^\ast)_{m=1}^\infty$.

Operator ideals are denoted by script letters such as $\opideal$. We denote by $\sepop$ the closed operator ideal consisting of all operators with separable range. The closed operator ideal $\sepop^\ast$, defined above as the class of all operators $T$ for which the adjoint $T^\ast$ belongs to $\sepop$, is a subclass of $\sepop$ \cite[Proposition~4.4.8]{Pietsch1980}. We denote by $\asplundop$ the operator ideal of {\it Asplund} operators, which consists of those operators $T$ for which $T|_Z\in\sepop^\ast$ for any separable subspace $Z$ of the domain of $T$.

For a set $S$ and a subset $R\subseteq S$ we write $\chi_R$ for the indicator function of $R$ in $S$ (with the underlying set $S$ always clear from the context). When discussing the Banach space $\ell_1(S)$ for some set $S$, for $s\in S$ we typically denote by $e_s$ the element of $\ell_1(S)$ satisfying $e_s(s')=1$ if $s'=s$ and $e_s(s') = 0$ if $s'\neq s$ ($s'\in S$). We thus denote by $(e_n)_{n=1}^\infty$ the standard unit vector basis of $\ell_1=\ell_1(\nat)$.

We shall repeatedly use the fact that for a set $I$, Banach space $X$ and family $\{ x_i\mid i\in I\}\subseteq X$ with $\sup_{i\in I}\Vert x_i\Vert <\infty$, there exists a unique element of $\allop(\ell_1(I), X)$ satisfying $e_i\mapsto x_i$, $i\in I$.

For a Banach space $X$, a subset $A\subseteq X$, and $\epsilon>0$, we say that $A$ is {\it $\epsilon$-separated} if $\Vert x-y\Vert>\epsilon$ for any distinct $x,y\in A$. For $B\subseteq C\subseteq X$ and $\delta>0$ we say that $B$ is a {\it $\delta$-net in $C$} if for every $w\in C$ there exists $z\in B$ such that $\Vert w-z\Vert \leq \delta$.

A {\it tree} is a partially ordered set $(\tee,\preceq)$ for which the set $\{ s\in\tee\mid s\preceq t\}$ is well-ordered for every $t\in\tee$. We shall frequently suppress the partial order $\preceq$ and refer to the underlying set $\tee$ as the tree. For $\ess\subseteq\tee$ we denote by $\MIN(\ess)$ the set of all minimal elements of $\ess$. 
A {\it subtree} of $\tee$ is a subset of $\tee$ equipped with the partial order induced by the partial order of $\tee$, which we also denote $\preceq$. A {\it chain} in $\tee$ is a totally ordered subset of $\tee$. A {\it branch} of $\tee$ is a maximal (with respect to set inclusion) totally ordered subset of $\tee$. We say that $\tee$ is {\it chain-complete} if every chain $\cee$ in $\tee$ admits a unique least upper bound. A subset $\ess\subseteq\tee$ is said to be {\it downwards closed in} $\tee$ if $\ess=\bigcup_{t\in\ess}\{ s\in\tee\mid s\preceq t\}$. For $s,t\in\tee$ we write $s\prec t$ to mean that $s\preceq t$ and $s\neq t$. For $t\in\tee$ we define the following sets:
\begin{align*}
\tee[\preceq t]&= \{ s\in\tee\mid s\preceq t\}\\
\tee[\prec t]&= \{ s\in\tee\mid s\prec t\}\\
\tee[t \preceq]&= \{ s\in\tee\mid t\preceq s\}\\
\tee[t\prec]&= \{ s\in\tee\mid t\prec s\}\\
\tee[t+]&= \MIN(\tee[t\prec])
\end{align*}
By $t^-$ we denote the maximal element of $\tee[\prec t]$, if it exists (that is, if the order type of $\tee[\prec t]$ is a successor). If $s,t\in\tee$ are such that $s\npreceq t$ and $t\npreceq s$, then we write $s\perp t$. 

Let $\tee= (\tee,\preceq)$ be a tree, $\alpha$ an ordinal and $\psi:\alpha\longrightarrow\tee$ a surjection. Then $\psi$ induces a well-ordering of $\tee$ that extends $\preceq$. Indeed, define $A_0 = \tee[\preceq \psi(0)]$ and, if $\beta>0$ is an ordinal such that $A_\gamma$ has been defined for all $\gamma<\beta$, define $A_\beta = \tee[\preceq \psi(\beta)]\setminus \bigcup_{\gamma<\beta}\tee[\preceq \psi(\gamma)]$. The induced well-order $\leq$ of $\tee$ is defined by declaring $s\leq t$, where $s\in A_\beta$ and $t\in A_{\beta'}$, if $\beta<\beta'$ or if $\beta=\beta'$ and $s\preceq t$. Notice that if $\tee$ is countable and $\tee[\prec t]$ is finite for every $t\in\tee$, then the well-ordering of $\tee$ induced as above by a surjection of $\omega$ onto $\tee$ is of order type $\omega$. In fact, the following statements are equivalent:
\begin{itemize}
\item[(i$_0$)] $\tee$ is countable and $\tee[\prec t]$ is finite for every $t\in\tee$;
\item[(ii$_0$)] There exists a bijection $\tau$ of $\omega$ onto $\tee$ such that $\tau(l)\preceq \tau(m)$ implies $l\leq m$ for $l,m<\omega$.
\end{itemize}

\begin{example}\label{fullcounter}
Let $\Omega:=\bigcup_{n<\omega}\prod_n\omega$. That is, $\Omega$ is the set of all finite (including possibly empty) sequences of finite ordinals. We define an order $\sqsubseteq$ on $\Omega$ by saying that $s\sqsubseteq t$ if and only if $s$ is an initial segment of $t$. For $n<\omega$ and $t\in\Omega$ we denote by $n^\smallfrown t$ the concatenation of $(n)$ with $t$; that is, $n^\smallfrown t=(n)$ if $t=\emptyset$ and $n^\smallfrown t=(n,n_1,\ldots,n_k)$ if $t=(n_1,\ldots,n_k)$. Notice that a tree is order-isomorphic to a subtree of $\Omega$ if and only if it satisfies the equivalent conditions (i$_0$) and (ii$_0$) above.
\end{example}

There are various natural topologies for trees, many of which are described in \cite{Nyikos1997}. The tree topology of interest to us is the {\it coarse wedge topology}, which is compact and Hausdorff for many trees. The coarse wedge topology of $(\tee,\preceq)$ is that topology on $\tee$ formed by taking as a subbase all sets of the form $\tee[t\preceq]$ and $\tee\setminus \tee[t\preceq]$, where the order type of $\tee[\prec t]$ is either $0$ or a successor ordinal. For a tree $(\tee,\preceq)$, $t\in \tee$ and $\eff\subseteq \tee$, define \[ \dubyoo_\tee(t,\eff):= \tee[t\preceq]\setminus \bigcup_{s\in \eff}\tee[s\preceq] \,.\] The following proposition is clear.
\begin{proposition}\label{coarsefacts}
Let $(\tee,\preceq)$ be a tree and let $t\in\tee$ be such that the order type of $\tee[\prec t]$ is $0$ or a successor ordinal. Then the coarse wedge topology of $\tee$ admits a local base of clopen sets at $t$ consisting of all sets of the form $\dubyoo_\tee(t,\eff)$, where $\eff\subseteq \tee[t+]$ is finite.
\end{proposition}

The following result is proved in the aforementioned paper of Nyikos.
\begin{theorem}\label{coarsecompact}(\cite[Corollary~3.5]{Nyikos1997})
Let $\tee$ be a tree. The following are equivalent.
\begin{itemize}
\item[(i)] $\tee$ is chain-complete and $\MIN(\tee)$ is finite.
\item[(ii)] The coarse wedge topology of $\tee$ is compact and Hausdorff.
\end{itemize}
\end{theorem}

\section{Non-Asplund Banach spaces and operators}\label{nonsnake}
In this section we derive our main technical result using techniques developed in \cite{PAHBSzlenkLarge}. We shall mainly work with downwards-closed subtrees of $(\Omega,\sqsubseteq)$, the tree of finite sequences of finite ordinals. The discussion preceding Example~\ref{fullcounter} alludes to the fact that results stated in terms of such trees can be expected to be readily extended to the class of countable trees $(\tee,\preceq)$ with $\tee[\prec t]$ finite for every $t\in\tee$. The advantage in restricting our attention to subtrees of $\Omega$ is a technical one, in particular that such trees do not contain any of their infinite branches as elements; this property of subtrees of $\Omega$ allows us to take the union of a tree $\tee$ with its set of infinite branches as a chain-completion of $\tee$ in a natural way which we now outline.

For $\tee$ a downwards-closed subtree of $\Omega$, we denote by $\partial\tee$ the set of all infinite branches of $\tee$ and define $\overline{\tee}:=\tee\cup\partial\tee$. Note that $\partial\tee\subseteq\partial\Omega$ for such $\tee$, hence $\overline{\tee}\subseteq\overline{\Omega}$. We extend the order $\sqsubseteq$ on $\Omega$ to an order $\sqsubseteq'$ on $\overline{\Omega}$ as follows: for $s,t\in\overline{\Omega}$ write $s\sqsubseteq' t$ if and only if one of the following conditions holds:
\begin{itemize}
\item[(i)] $s,t\in\Omega$ and $s\sqsubseteq t$;
\item[(ii)] $s\in t\in\partial\Omega$; and,
\item[(iii)] $s=t\in \partial\Omega$.
\end{itemize}
The fact that $\Omega\cap\partial\Omega=\emptyset$ ensures that $\sqsubseteq'$ is well-defined partial order making $\overline{\Omega}$ a tree. Moreover it is straightforward to check that $(\overline{\tee},\sqsubseteq')$ is a chain-complete tree for any downwards-closed $\tee\subseteq\Omega$. It follows by Theorem~\ref{coarsecompact} that the coarse wedge topology of $\overline{\tee}$ is compact and Hausdorff.

The following theorem is the main result of the current paper.

\begin{theorem}\label{scoopie}
Let $W$ and $Z$ be Banach spaces such that $Z$ is separable. Suppose $R\in\allop(W,Z)$, $K\subseteq Z^\ast$ and $\epsilon>0$ are such that $K$ is absolutely convex and weak$^\ast$-compact and $R^\ast (K)$ contains an uncountable $\epsilon$-separated subset. Then for any $\delta>0$, any $\theta>0$, any downwards-closed $\tee\subseteq\Omega$, and any bijection $\tau:\omega\longrightarrow \tee$ such that $\tau(l)\sqsubseteq\tau(m)$ implies $l\leq m$, there exist families $(w_t)_{t\in\tee}\subseteq S_W$ and $(z_t^\ast)_{t\in\overline{\tee}}\subseteq K$ such that:
\begin{itemize}
\item[(i)] \begin{equation}\label{biorslogger}
\langle z_t^\ast ,Rw_s\rangle=\begin{cases}
\langle z_s^\ast,Rw_s\rangle>\frac{\epsilon}{2+\theta}&\mbox{if }s\sqsubseteq' t\\ 0 &\mbox{if }s\not\sqsubseteq' t
\end{cases}, \qquad s\in\tee,\, t\in\overline{\tee};
\end{equation}
\item[(ii)] The map $t\mapsto z_t^\ast$ from $\overline{\tee}$ to $Z^\ast$ is coarse-wedge-to-weak$^\ast$ continuous;
\item[(iii)] $(w_{\tau(m)})_{m=0}^\infty$ is a basic sequence with basis constant not exceeding $1+\delta$; and,
\item[(iv)] With $Z_0: = \bigcap_{t\in\tee}\ker(z_t^\ast)$ and $Q:Z\longrightarrow Z/Z_0$ denoting the quotient map, $(QRw_{\tau(m)})_{m=0}^\infty$ is a basis for $Z/Z_0$ with basis constant not exceeding $1+\delta$.
\end{itemize}
\end{theorem}

Notice that (\ref{biorslogger}) implies that the map $t\mapsto z_t^\ast$ in Theorem~\ref{scoopie} is one-to-one, and therefore a coarse-wedge-to-weak$^\ast$ homeomorphism. In the particular case that $\tee=\dee$ (the infinite dyadic tree of finite sequences in $\{0,1\}$) we have that $\partial\dee$ is closed in $\overline{\dee}$ (since $\dee[t+]$ is finite for every $t\in\dee$) and homeomorphic to $\{ 0,1\}^\omega$. Indeed, the map from $\{ 0,1\}^\omega$ to $\partial\dee$ that maps each $s\in\{ 0,1\}^\omega$ to the set $\{ s|_n\mid n<\omega\}$ of all (finite) initial segments of $s$ is such a homeomorphism. In particular, and in a similar vein to Theorem~\ref{doodah}, Theorem~\ref{scoopie} provides a construction of a norm discrete, weak$^\ast$-homeomorphic copy of the Cantor set inside $K$. It is equality (rather than approximate equality with error depending on $s$) at (\ref{biorslogger}) that allows us to obtain a version of Theorem~\ref{doodah} with $\epsilon=0$.

Our proof of Theorem~\ref{scoopie} requires the following result, which is Lemma~2.2 of \cite{Dilworth2016}.
\begin{lemma}\label{lowerlemma}
Let $X$ be a Banach space, $\nu>0$ a real number, $F$ a finite dimensional subspace of $X^\ast$, $A$ a $\frac{\nu}{4+2\nu}$-net in $S_F$ and $\{ y_{f^\ast}\mid f^\ast\in A\}\subseteq S_X$ a family such that $\inf\{ \vert f^\ast(y_{f^\ast})\vert\mid f^\ast\in A\}\geq \frac{4+\nu}{4+2\nu}$. Then for every $x^\ast\in \{ y_{f^\ast}\mid f^\ast\in A\}^\perp$ we have $ \sup\{ \vert x^\ast (y)\vert\mid y\in B_{F_\perp} \}\geq \frac{1}{2+\nu}\Vert x^\ast\Vert$.
\end{lemma}

The following result seems to be folklore; we refer the reader to \cite{PAHBSzlenkLarge} for a proof.
\begin{lemma}\label{zippinbase}
Let $X$ be a Banach space and $x^\ast\in X^\ast$. Then \[ \{ x^\ast + \epsilon B_{X^\ast}^\circ + C^\perp \mid \epsilon>0,\, C\subseteq X, \, \vert C\vert <\infty\} \] is a local base for the weak$^\ast$ topology of $X^\ast$ at $x^\ast$.
\end{lemma}

\begin{proof}[Proof of Theorem~\ref{scoopie}]
Fix $\delta>0$, $\theta>0$ and $\tau:\omega\longrightarrow \tee$ a bijection such that $\tau(l)\sqsubseteq\tau(m)$ implies $l\leq m$. Let $L'\subseteq K$ be uncountable and such that $R^\ast(L')$ is $\epsilon$-separated. We shall use small perturbations to construct a family $(f_t^\ast)_{t\in\tee}\subseteq Z^\ast$, weak$^\ast$-homeomorphic to $\tee$ with its coarse wedge topology, that `almost' lies inside $L'$ and which satisfies a condition similar to (\ref{biorslogger}). Points $f_t^\ast$, $t\in\partial\tee$, are then defined by continuous extension of the map $t\mapsto f_t^\ast$ to all of $\overline{\tee}$. The family $(z_t^\ast)_{t\in\overline{\tee}}$ in the statement of the theorem is then obtained by multiplying $(f_t^\ast)_{t\in\overline{\tee}}$ pointwise by a suitable scalar. Lemma~\ref{lowerlemma} and Lemma~\ref{zippinbase} above will play a key role in establishing the existence of the desired family $(w_t)_{t\in\tee}\subseteq W$ from the statement of the theorem. Our construction of the families $(f_t^\ast)_{t\in\tee}$ and $(w_t)_{t\in\tee}$ will be by induction with respect to the well-ordering of $\tee$ induced by $\tau$. 

We now make some preliminary definitions and observations that are required in order to state the key construction of the proof in full detail at (I)-(VI) below. Removing one point from $L'$ if necessary, we may assume that $\Vert R^\ast z^\ast\Vert >\epsilon/2$ for every $z^\ast\in L'$. By Theorem~8.5.2 of \cite{Semadeni1971} we may write $L'=L\cup L''$, where $L$ is weak$^\ast$-dense-in-itself and $L''$ is scattered in the weak$^\ast$ topology. Set $\nu=\theta(10 +3\theta)^{-1}$, so that \[\frac{1-3\nu}{(2+\nu)(1+\nu)}\geq\frac{1}{2+\theta},\] and let $(\delta_m)_{m=0}^\infty\subseteq(0,1)$ be a sequence of positive real numbers such that $\sum_{m=0}^\infty\delta_m<\infty$ and $\prod_{m=0}^\infty(1-\delta_m)\geq(1+\delta)^{-1}$.
Let $(z_p)_{p=0}^\infty$ be a norm-dense sequence in $S_Z$ and let $d$ be the metric on $Z^\ast$ given by setting $d(y^\ast,z^\ast)=\sum_{p=0}^\infty 2^{-p-1}\vert \langle y^\ast-z^\ast,z_p\rangle\vert$ for $y^\ast,z^\ast\in Z^\ast$. As is well-known (see, e.g., Proposition~3.22 of \cite{Fabian2001}), on bounded subsets of $Z^\ast$ the topology induced by $d$ coincides with the weak$^\ast$ topology. Moreover, $d(y^\ast,z^\ast)\leq\Vert y^\ast-z^\ast\Vert$ for all $y^\ast,z^\ast\in Z^\ast$. 

We now outline the key construction of the proof. Proceeding by induction over $m<\omega$, we shall construct families $(f_{\tau(m)}^\ast)_{m<\omega}\subseteq (1+\nu)K$ and $(w_{\tau(m)})_{m<\omega}\subseteq S_W$ and a strictly increasing sequence $(p_m)_{m=0}^\infty$ in $\omega$ so that, with $v_m:= \Vert f_{\tau(m)}^\ast - f_{\tau(m)^-}^\ast\Vert^{-1}(f_{\tau(m)}^\ast - f_{\tau(m)^-}^\ast)$ for each $m\in[1,\omega)$ and with $v_0:=f_{\tau(0)}^\ast$, the following conditions hold for all $m<\omega$:
\begin{itemize}
\item[(I)]
$f_{\tau(m)}^\ast\in(\nu\sum_{j=1}^m 2^{-j})K+L\subseteq (1+\nu)K$;
\item[(II)] If $m>0$ then $d(f_{\tau(m)}^\ast,f_{\tau(m)^-}^\ast)<2^{-m}$;
\item[(III)] For all $i,j\leq m$,
\begin{equation}\label{mutellanan}
\langle f_{\tau(j)}^\ast ,Rw_{\tau(i)}\rangle=\begin{cases}
\langle f_{\tau(i)}^\ast,Rw_{\tau(i)}\rangle
>\frac{(1-3\nu)\epsilon}{2+\nu}&\mbox{if }{\tau(i)}\sqsubseteq {\tau(j)}\\ 0 &\mbox{if }{\tau(i)}\not\sqsubseteq {\tau(j)}
\end{cases};
\end{equation}
\item[(IV)] For all $w\in \spn\{ w_{\tau(j)}\mid 0\leq j< m\}$ and scalars $a$ we have $\Vert w+ a w_{\tau(m)}\Vert \geq (1-\delta_m)\Vert w\Vert$;
\item[(V)] For each $v^\ast \in [(v_j)_{j=0}^m]^\ast$ with $\Vert v^\ast\Vert \leq1$ there is a natural number $p\leq p_m$ such that $\vert \langle v,z_p\rangle - \langle v^\ast,v\rangle\vert \leq \delta_m/3$ for all $v\in [(v_j)_{j=0}^m]$; and,
\item[(VI)] If $m>0$ then $\vert \langle v_m, z_p\rangle\vert <\delta_m/3$ for all $p\leq p_{m-1}$. 
\end{itemize}

We defer the construction of the families $(f_{\tau(m)}^\ast)_{m<\omega}$, $(w_{\tau(m)})_{m<\omega}$ and $(p_m)_{m=0}^\infty$ and will now show how to prove Theorem~\ref{scoopie} under the assumption that the families $(f_{\tau(m)}^\ast)_{m<\omega}$, $(w_{\tau(m)})_{m<\omega}$ and $(p_m)_{m=0}^\infty$ satisfying (I)-(VI) have already been constructed. Our first step is to define the family $(z_t^\ast)_{t\in\overline{\tee}}\subseteq  K$. To this end for each $b\in \partial\tee$ and $n<\omega$ let $b_n$ denote the element of $b$ for which the order type of $\tee[\sqsubset b_n]$ is $n$ and let $f_b^\ast$ denote the weak$^\ast$ limit of the sequence $(f_{b_n}^\ast)_{n=0}^\infty\subseteq (1+\nu) K$, which is Cauchy with respect to $d$ since (II) holds for all $m<\omega$. For $t\in\tee$ and $b\in\partial\tee$ it follows from the fact that (III) holds for all $m<\omega$ that
\begin{equation}\label{likegreece}
\langle f_b^\ast, Rw_t\rangle = \begin{cases}
\langle f_t^\ast, Rw_t\rangle>\frac{(1-3\nu)\epsilon}{2+\nu},&\mbox{if }t\sqsubseteq' b\\
0,&\mbox{if }t\not\sqsubseteq'b
\end{cases}.
\end{equation}
Let $z_t^\ast=\frac{1}{1+\nu} f_t^\ast$ for each $t\in\overline{\tee}$. Then (\ref{biorslogger}) follows from (\ref{mutellanan}) and (\ref{likegreece}); in particular, assertion (i) of the theorem holds. 

Let $\Xi:\overline{\tee}\longrightarrow Z^\ast$ be the map given by setting $\Xi(t)=z_t^\ast$ for each $t\in\overline{\tee}$. To prove (ii) we want to show that $\Xi$ is coarse-wedge-to-weak$^\ast$ continuous. To this end first note that for $0<m<\omega$ we have $d(z_{\tau(m)}^\ast,z_{\tau(m)^-}^\ast)<2^{-m}/(1+\nu)<2^{-m}$ by (II). For $m,m'<\omega$ satisfying $\tau(m)\sqsubset\tau(m')$ we have $m< m'$, so by (II) it follows that for $m<\omega$ and $t\in\tee$ such that $\tau(m)\sqsubset t$ we have
\[
d(z_{\tau(m)}^\ast,z_t^\ast)\leq \sum_{s\in (\tau(m), t]}d(z_s^\ast,z_{s^-}^\ast)< \sum_{m''=m+1}^\infty2^{-m''}\leq 2^{-m}.
\]
Hence for $m<\omega$ and $b\in\partial\tee$ satisfying $\tau(m)\in b$ we have $d(z_{\tau(m)}^\ast,z_b^\ast)\leq2^{-m}$ since $z_b^\ast$ is the weak$^\ast$ limit of the sequence $(z_{b_n}^\ast)_{n=0}^\infty$, all but finitely many terms of which satisfy $d(z_{\tau(m)}^\ast,z_{b_n}^\ast)\leq2^{-m}$.

We are now ready to show that $\Xi$ is continuous at each point of $\overline{\tee}$. Fix $\lambda>0$ and let $t\in\overline{\tee}$. Let us first suppose that $t\in\tee$. Let $\eff\subseteq\tee[t+]$ be a finite set such that $s'\in\tee[t+]\setminus\eff$ implies $2^{-\tau^{-1}(s')+1}<\lambda$. For each $s\in\dubyoo_{\overline{\tee}}(t,\eff)\setminus\{ t\}$ let $s'$ denote the unique element of $\tee[t+]\cap\tee[\sqsubseteq s]$. Since $\dubyoo_{\overline{\tee}}(t,\eff)$ is a coarse-wedge neighbourhood of $t$ in $\overline{\tee}$ and since for every $s\in \dubyoo_{\overline{\tee}}(t,\eff)\setminus \{ t\}$, we have\[
d(z_t^\ast,z_s^\ast)\leq d(z_t^\ast,z_{s'}^\ast)+d(z_{s'}^\ast,z_s^\ast)< 2^{-\tau^{-1}(s')} +2^{-\tau^{-1}(s')} = 2^{-\tau^{-1}(s')+1}<\lambda,
\] we conclude that $\Xi$ is continuous at $t$ in this case. 

Now suppose that $b\in\partial\tee$ and let $m_0<\omega$ be so large that $2^{-m_0+1}<\lambda$. By definition, $\dubyoo_{\overline{\tee}}(b_{m_0},\emptyset)$ is a coarse wedge neighbourhood of $b$. Since $\tau^{-1}(b_{m_0})\geq m_0$ we have, for $s\in \dubyoo_{\overline{\tee}}(b_{m_0},\emptyset)$,
\[
d(z_b^\ast,z_s^\ast)\leq d(z_b^\ast,z_{b_{m_0}}^\ast)+d(z_{b_{m_0}}^\ast,z_s^\ast)\leq2^{-\tau^{-1}(b_{m_0})}+2^{-\tau^{-1}(b_{m_0})}\leq 2^{-m_0+1}<\lambda,
\]
hence $\Xi$ is continuous at $b$. We have now established the desired continuity of $\Xi$, so that assertion (ii) of the theorem is proved.

Assertion (iii) of the theorem follows from the Grunblum criterion and the fact that, since (IV) holds for all $m<\omega$, for $0\leq l\leq m<\omega$ and scalars $a_0,a_1,\ldots,a_m$ we have
\[
\Big\Vert \sum_{q=0}^la_qw_{\tau(q)}\Big\Vert \leq \frac{1}{\displaystyle\prod_{q=l+1}^m(1-\delta_q)}\Big\Vert \sum_{q=0}^ma_qw_{\tau(q)}\Big\Vert
\leq (1+\delta)\Big\Vert \sum_{q=0}^ma_qw_{\tau(q)}\Big\Vert\,.
\]

The main idea in the proof of (iv) is to apply the techniques set out in Johnson and Rosenthal's proof of Theorem~III.1 of \cite{Johnson1972}. Our first step is to show that $(v_m)_{m=1}^\infty$ is a basic sequence with basis constant no larger than $1+\delta$. To this end fix $m\in[1,\omega)$ and suppose $v\in[(v_q)_{q=1}^m]$ is such that $\Vert  v\Vert =1$. Choose $v^\ast\in[v_j]_{1\leq j\leq m}^\ast$ such that $\langle v^\ast ,v\rangle=1=\Vert v^\ast\Vert$ and choose $p\leq p_m$ so that (V) holds for $v^\ast$. Then $\vert\langle v,z_p\rangle\vert\geq 1-\delta_m/3$, hence for any scalar $a$ we have \begin{align*}
\Vert v+a v_{m+1}\Vert &\begin{cases} >1 & \mbox{if } \vert a\vert >2\\\geq \vert \langle v ,z_p\rangle + \langle a v_{m+1},z_p\rangle\vert\geq (1-\frac{\delta_m}{3})-\frac{2\delta_m}{3}& \mbox{if }\vert a\vert\leq2
\end{cases}\\
&\geq 1-\delta_m.
\end{align*}
It follows that $\Vert\sum_{q=1}^ma_qv_q\Vert\leq \frac{1}{1-\delta_m}\Vert \sum_{q=1}^{m+1}a_qv_q\Vert $ for any scalars $a_1,\ldots,a_m,a_{m+1}$. Thus for $1\leq l\leq m<\omega$ and any scalars $a_1,\ldots,a_m$ we have
\begin{equation}\label{steefryle}
\Big\Vert \sum_{q=1}^la_qv_q\Big\Vert\leq \frac{1}{\displaystyle\prod_{q=l+1}^m(1-\delta_q)}\Big\Vert \sum_{q=1}^ma_qv_q\Big\Vert
\leq (1+\delta)\Big\Vert \sum_{q=1}^ma_qv_q\Big\Vert.
\end{equation}
By Grunblum's criterion, $(v_m)_{m=1}^\infty$ is a basic sequence with basis constant no larger than $1+\delta$.

Let $(v_m^\ast)_{m=1}^\infty$ be the sequence of functionals in $[(v_m)_{m=1}^\infty]^\ast$ biorthogonal to $(v_m)_{m=1}^\infty$ and define $T: Z\longrightarrow [(v_m)_{m=1}^\infty]^\ast$ by $\langle Tz,v\rangle=\langle v,z\rangle$ for $z\in Z$ and $v\in[(v_m)_{m=1}^\infty]$. That is, $Tz=(\imath_Zz)|_{[(v_m)_{m=1}^\infty]}$ for each $z\in Z$. Note that $\ker(T)=\bigcap_{m=1}^\infty\ker(v_m)$. For convenience let $f_{\tau(0)^-}^\ast $ be defined to be the zero element of $Z^\ast$ (notwithstanding the fact that $\tau(0)^-$ is undefined), so that for each $t\in\tee$ we have \[ f_t^\ast=\sum_{s\sqsubseteq t}(f_s^\ast-f_{s^-}^\ast).\] Then we have
\[
\ker(T)=\bigcap_{m=1}^\infty\ker(v_m) = \bigcap_{t\in\Omega}\ker(f_t^\ast-f_{t^-}^\ast)=\bigcap_{t\in\Omega}\ker(f_t^\ast)=\bigcap_{t\in\Omega}\ker(z_t^\ast)=Z_0.
\] 
Following the argument in the proof of Theorem~III.1 of \cite{Johnson1972} yields the equality $T(Z) = [(v_m^\ast)_{m=1}^\infty]$ and the existence of a linear isometry $\overline{T}:Z/Z_0\longrightarrow  [(v_m^\ast)_{m=1}^\infty]$ such that $\overline{T}Qz=Tz$ for every $z\in Z$. By Fact 6.6 of \cite{Fabian2001}, for $m\in[1,\omega)$ we have
\begin{align*}
\overline{T}QRw_{\tau(m)} &=TRw_{\tau(m)}\\& = \sum_{m'=1}^\infty\langle TRw_{\tau(m)},v_{m'}\rangle v_{m'}^\ast\\&= \sum_{m'=1}^\infty\langle v_{m'},Rw_{\tau(m)}\rangle v_{m'}^\ast\\& =\frac{\langle f_{\tau(m)}^\ast,Rw_{\tau(m)}\rangle}{\Vert f_{\tau(m)}^\ast- f_{\tau(m)^-}^\ast \Vert}v_m^\ast.
\end{align*}
For each $m\in[1,\omega)$ let $a_m= \langle f_{\tau(m)}^\ast,Rw_{\tau(m)}\rangle\Vert f_{\tau(m)}^\ast- f_{\tau(m)^-}^\ast \Vert^{-1}$, so that $\overline{T}QRw_{\tau(m)}=a_mv_m^\ast$ for each such $m$. As $\overline{T}$ is an isometry, $(QRw_{\tau(m)})_{m=1}^\infty$ is a basis for $Z/Z_0$ isometrically equivalent to $(a_mv_m^\ast)$, whose basis constant coincides with the basis constant of $(v_m^\ast)_{m=1}^\infty$, which coincides with the basis constant of $(v_m)_{m=1}^\infty$, which is no larger than $1+\delta$ (as shown above). This completes the proof of (iv), and so to complete the proof of Theorem~\ref{scoopie} it remains only to carry out the inductive construction of the families $(f_{\tau(m)}^\ast)_{m<\omega}$, $(w_{\tau(m)})_{m<\omega}$ and $(p_m)_{m=0}^\infty$ so that (I)-(VI) above hold for all $m<\omega$. 

We begin the induction by letting $f_{\tau(0)}^\ast$ be an arbitrary point in $L$ and choosing $w_{\tau(0)}\in S_W$ such that $\langle R^\ast f_{\tau(0)}^\ast ,w_{\tau(0)}\rangle >\epsilon/2$. Define $v_0:= f_{\tau(0)}^\ast$. By Helly's theorem, by the density of $(z_p)_{p=1}^\infty$ in $S_Z$, and by the total boundedness of $S_{[\{v_0\}]}$ and $S_{[\{v_0\}]^\ast}$, we may choose $p_0<\omega$ large enough that (V) holds for $m=0$. It is now easily checked that (I)-(VI) hold for $m=0$.

Fix $k<\omega$ and suppose that $f_{\tau(m)}^\ast\in(\nu\sum_{j=1}^m 2^{-j})K+L$, $w_{\tau(m)}\in S_W$ and $p_m<\omega$ have been defined for $1\leq m\leq k$ in such a way that $(p_m)_{m=1}^k$ is strictly increasing and properties (I)-(VI) are satisfied for $0\leq m\leq k$. To carry out the inductive step of the proof, we now show how to construct $f_{\tau(k+1)}^\ast\in(\nu\sum_{j=1}^{k+1} 2^{-j})K+L$, $w_{\tau(k+1)}\in S_W$ and $p_{k+1}\in (p_k,\omega)$ so that (I)-(VI) are satisfied for $0\leq m\leq k+1$. To this end let $v^\ast\in (\nu\sum_{j=1}^{k} 2^{-j})K$ and $z_1^\ast\in L$ be such that $f_{\tau(k+1)^-}^\ast=v^\ast+z_1^\ast$. With a view to applying Lemma~\ref{lowerlemma}, let $G$ be a finite $\delta_{k+1}$-net in $S_{\spn\{ w_{\tau(i)}\mid 0\leq i\leq k \}}$ and for each $g\in G$ let $h_g^\ast\in W^\ast$ be such that $\langle h_g^\ast ,g\rangle=1$. Set \[ F= \spn\big( \{ R^\ast f_{\tau(j)}^\ast\mid0\leq j\leq k \} \cup \{ h_g^\ast \mid g\in G\} \big)\subseteq W^\ast,\] let $A$ be a finite $\frac{\nu}{4+2\nu}$-net in $S_F$ and let $\{ y_{f^\ast}\mid f^\ast\in A \}\subseteq S_W$ be such that $\langle f^\ast,y_{f^\ast}\rangle\geq \frac{4+\nu}{4+2\nu}$ for each $f^\ast\in A$. Let $e\in\real$ be such that
\[
ke \frac{(2+\nu)}{(1-3\nu)\epsilon}2(1+\nu)\sup\{ \Vert z^\ast\Vert \mid z^\ast \in K\}=\frac{1}{2} \frac{\delta_{k+1}}{3}\frac{(1-3\nu)\epsilon}{(2+\nu)\Vert R\Vert}
\]
and define
\begin{align*}
\yoo_1&:=\bigcap_{i=0}^k\Big\{ z^\ast\in Z^\ast\,\,\Big\vert\,\, \vert \langle z^\ast - z_1^\ast,Rw_{\tau(i)}\rangle\vert<\frac{2^{-k-2}\nu(1-3\nu)\epsilon}{k(2+\nu)(1+\nu)} \Big\};\\
\yoo_2&:=\bigcap_{i=0}^k\Big\{ z^\ast\in Z^\ast\,\,\Big\vert\,\, \vert \langle z^\ast - z_1^\ast,Rw_{\tau(i)}\rangle\vert<\frac{2^{-k-3}(1-3\nu)\epsilon}{k(2+\nu)(1+\nu)\sup_{z^\ast\in K}\Vert z^\ast\Vert} \Big\};\\
\yoo_3&:=\bigcap_{i=0}^k\{ z^\ast\in Z^\ast\mid \vert \langle z^\ast - z_1^\ast,Rw_{\tau(i)}\rangle\vert<e \};\\
\yoo_4&:=\bigcap_{p=0}^{p_k}\Big\{ z^\ast\in Z^\ast\,\,\Big\vert\,\, \vert \langle z^\ast - z_1^\ast,z_p\rangle\vert< \frac{1}{2}\frac{\delta_{k+1}}{3}\frac{(1-3\nu)\epsilon}{(2+\nu)\Vert R\Vert} \Big\};\\
\yoo_5&:=\{ z^\ast \in Z^\ast\mid d(z^\ast, z_1^\ast)<2^{-k-2}\};\\
\yoo_6&:=\big\{ z^\ast \in Z^\ast \mid R^\ast z^\ast\in R^\ast z_1^\ast + \frac{\epsilon\nu}{2+\nu}B_{W^\ast}^\circ+ \big(\{ w_{\tau(i)}\mid0\leq i\leq k \}\cup\{ y_{f^\ast}\mid f^\ast\in A\}\big)^\perp\big\};\mbox{ and,}\\
\yoo&:=\yoo_1\cap\yoo_2\cap\yoo_3\cap\yoo_4\cap\yoo_5\cap\yoo_6.
\end{align*}
$L\cap\yoo$ is a weak$^\ast$-neighbourhood of $z_1^\ast$ in $L$, hence $(L\cap\yoo)\setminus \{ z_1^\ast\}$ is nonempty since $L$ is weak$^\ast$-dense-in-itself. Choose $z_2^\ast \in (L\cap\yoo)\setminus \{ z_1^\ast\}$. We will ultimately use the fact that $z_2^\ast$ belongs to each of the neighbourhoods $\yoo_1,\dots,\yoo_6$ to show that (I)-(VI) are satisfied for $0\leq m\leq k+1$ with a suitable choice for $f_{\tau(k+1)}^\ast$.

Set $u^\ast=v^\ast +z_2^\ast$, so that $u^\ast\in (\nu\sum_{j=1}^k 2^{-j})K+L$. Since $z_2^\ast\in\yoo_6$ we may write $R^\ast z_2^\ast = R^\ast z_1^\ast+ y^\ast +x^\ast$, where $\Vert y^\ast\Vert <\nu\epsilon/(2+\nu)$ and \[x^\ast \in (\{ w_{\tau(i)}\mid 0\leq i\leq k \}\cup\{ y_{f^\ast}\mid f^\ast\in A\})^\perp.\] We then have \[ R^\ast u^\ast = R^\ast v^\ast+R^\ast z_2^\ast = R^\ast v^\ast+ R^\ast z_1^\ast+ y^\ast +x^\ast = R^\ast f_{\tau(k+1)^-}^\ast+ y^\ast +x^\ast.\] Since
\[
\Vert x^\ast\Vert \geq \Vert R^\ast u^\ast - R^\ast f_{\tau(k+1)^-}^\ast\Vert - \Vert y^\ast\Vert = \Vert R^\ast z_2^\ast - R^\ast z_1^\ast\Vert - \Vert y^\ast\Vert  > \epsilon-\frac{\nu\epsilon}{2+\nu} >(1-\nu)\epsilon,
\]
an application of Lemma~\ref{lowerlemma} with $F$, $A$ and $\{ y_{f^\ast}\mid f^\ast\in A \}\subseteq S_W$ as defined above in the current proof yields $y\in S_{F_\perp}$ such that 
\[
\langle x^\ast ,y\rangle > \frac{1}{2+\nu}\cdot (1-\nu)\epsilon- \frac{\nu\epsilon}{2+\nu}= \frac{(1-2\nu)\epsilon}{2+\nu}.
\]
Since $y\in\ker(R^\ast f_{\tau(k+1)^-}^\ast)$ we have
\[
\vert \langle R^\ast u^\ast, y\rangle\vert = \vert \langle R^\ast u^\ast-R^\ast f_{\tau(k+1)^-}^\ast, y\rangle\vert = \vert \langle x^\ast +y^\ast, y\rangle\vert > \frac{(1-2\nu)\epsilon}{2+\nu} - \frac{\nu\epsilon}{2+\nu} =\frac{(1-3\nu)\epsilon}{2+\nu},
\]
hence $w_{\tau(k+1)}:= \vert \langle R^\ast u^\ast,y\rangle\vert \langle R^\ast u^\ast,y\rangle^{-1}y$ is well-defined, $w_{\tau(k+1)}\in S_{F_\perp}\subseteq S_W$, and 
\begin{equation}\label{pickaloo}
\langle u^\ast, Rw_{\tau(k+1)}\rangle=\langle R^\ast u^\ast, w_{\tau(k+1)}\rangle= \vert \langle R^\ast u^\ast, y\rangle\vert >\frac{(1-3\nu)\epsilon}{2+\nu}.
\end{equation}

We now define $f_{\tau(k+1)}^\ast$, recalling that $f_{\tau(0)^-}^\ast$ is defined above as the zero element of $Z^\ast$. Define
\begin{align}
f_{\tau(k+1)}^\ast:=\label{coffee1} u^\ast\, -&\sum_{l=0}^k \frac{\langle u^\ast -f_{\tau(k+1)^-}^\ast,Rw_{\tau(l)}\rangle}{\langle f_{\tau(l)}^\ast,Rw_{\tau(l)}\rangle}(f_{\tau(l)}^\ast -f_{\tau(l)^-}^\ast)\\ \label{coffee2}= u^\ast\, -&\sum_{l=0}^k \frac{\langle z_2^\ast-z_1^\ast,Rw_{\tau(l)}\rangle}{\langle f_{\tau(l)}^\ast,Rw_{\tau(l)}\rangle}(f_{\tau(l)}^\ast -f_{\tau(l)^-}^\ast),
\end{align}
noting that the equality of (\ref{coffee1}) and (\ref{coffee2}) follows from the fact that $u^\ast-f_{\tau(k+1)^-}^\ast=z_2^\ast-z_1^\ast$.

Since $u^\ast\in (\nu\sum_{j=1}^k 2^{-j})K+L$ we have $f_{\tau(k+1)}^\ast\in ((\nu\sum_{j=1}^k 2^{-j})K+L)+cK$ for some scalar $c>0$. Since $z_2^\ast \in \yoo_1$, it follows from the definition of $f_{\tau(k+1)}^\ast$ and the assumption that (I) and (III) hold for $m\leq k$ that the scalar $c$ may be taken to satisfy \[ c\leq k\frac{2^{-k-2}\nu(1-3\nu)\epsilon}{k(2+\nu)(1+\nu)}\Bigg(\frac{(1-3\nu)\epsilon}{2+\nu} \Bigg)^{-1}2(1+\nu)= \nu2^{-k-1}.\]
In particular, we have $f_{\tau(k+1)}^\ast\in(\nu\sum_{j=1}^{k+1} 2^{-j})K+L$, so that (I) holds for $m=k+1$.

To see that (II) holds for $m=k+1$, observe that, since $z_2^\ast \in \yoo_2$, a similar argument to that used to estimate the scalar $c$ in the previous paragraph yields
\[
\Vert f_{\tau(k+1)}^\ast-u^\ast\Vert  \leq \sum_{0\leq l\leq k} \frac{\vert \langle z_2^\ast-z_1^\ast,Rw_{\tau(l)}\rangle\vert}{\vert \langle f_{\tau(l)}^\ast ,Rw_{\tau(l)}\rangle\vert}2(1+\nu)\sup_{z^\ast\in K}\Vert z^\ast\Vert\leq2^{-k-2}.
\]
It follows from this estimate and the fact that $z_2^\ast \in\yoo_5$ that
\begin{align*}
d(f_{\tau(k+1)}^\ast,f_{\tau(k+1)^-}^\ast)&\leq d(f_{\tau(k+1)}^\ast,u^\ast) + d(u^\ast,f_{\tau(k+1)^-}^\ast)\\
&\leq \Vert f_{\tau(k+1)}^\ast-u^\ast\Vert + d(u^\ast-v^\ast,f_{\tau(k+1)^-}^\ast-v^\ast)\\
&\leq 2^{-k-2}+d(z_2^\ast,z_1^\ast)\\
&<2^{-k-2}+2^{-k-2}\\
&=2^{-k-1}.
\end{align*}
That is, (II) holds for $m=k+1$.

We now show that (III) holds for $m=k+1$. By the induction hypothesis, we need only prove the case where at least one of $i$ and $j$ is equal to $k+1$. Since $w_{\tau(k+1)}\in S_{F_\perp}\subseteq \bigcap_{j=0}^k\ker(R^\ast f_{\tau(j)}^\ast)$ we have $\langle f_{\tau(j)}^\ast,Rw_{\tau(k+1)}\rangle=0$ for $0\leq j\leq k$. Moreover, by the definition of $f_{\tau(k+1)}^\ast$, by (\ref{pickaloo}) and by the fact that $w_{\tau(k+1)}\in \bigcap_{j=0}^k\ker(R^\ast f_{\tau(j)}^\ast)$, we have 
\[
\langle f_{\tau(k+1)}^\ast,Rw_{\tau(k+1)}\rangle= \langle u^\ast,Rw_{\tau(k+1)}\rangle-0= \langle u^\ast,Rw_{\tau(k+1)}\rangle >\frac{(1-3\nu)\epsilon}{2+\nu}.\]
Since (III) holds for $0\leq m\leq k$, if $l,i\in[0,k]$ then
\begin{equation*}
\langle f_{\tau(l)}^\ast - f_{\tau(l)^-}^\ast,Rw_{\tau(i)}\rangle=\begin{cases} \langle f_{\tau(i)}^\ast ,Rw_{\tau(i)}\rangle,&\mbox{if }i=l\\0,&\mbox{if }i\neq l \end{cases}.
\end{equation*}
It follows that if $0\leq i\leq k$ then
\begin{align*}
\langle f_{\tau(k+1)}^\ast, Rw_{\tau(i)}\rangle&= \langle u^\ast,Rw_{\tau(i)}\rangle - \frac{\langle u^\ast-f_{\tau(k+1)^-}^\ast,Rw_{\tau(i)}\rangle}{\langle f_{\tau(i)}^\ast,Rw_{\tau(i)}\rangle}\langle f_{\tau(i)}^\ast ,Rw_{\tau(i)}\rangle \notag \\&=\langle f_{\tau(k+1)^-}^\ast,Rw_{\tau(i)}\rangle\\
&=\begin{cases}
\langle f_{\tau(i)}^\ast,Rw_{\tau(i)}\rangle >\frac{(1-3\nu)\epsilon}{2+\nu}& \mbox{if }\tau(i) \sqsubseteq' \tau(k+1)\\
0&\mbox{if }\tau(i) \perp \tau(k+1)
\end{cases}.
\end{align*}
We have now shown that (III) holds for $m=k+1$.

Next we show that (IV) holds for $m=k+1$. Let $a$ be a scalar and $w\in \spn\{ w_{\tau(i)}\mid 0\leq i\leq k\} $. To avoid triviality, assume $w\neq 0$. Let $g_w\in G$ be such that $\Vert \Vert w\Vert^{-1}w-g_w\Vert \leq \delta_{k+1}$. Since $\langle h_{g_w}^\ast, g_w\rangle = 1$ and $w_{\tau(k+1)}\in\ker(h_{g_w}^\ast)$ we have
\begin{align*}
\Vert w+a w_{\tau(k+1)}\Vert &\geq \vert \langle h_{g_w}^\ast, w+a w_{\tau(k+1)}\rangle\vert \\&= \vert \langle h_{g_w}^\ast, \Vert w\Vert^{-1}w\rangle\vert \Vert w\Vert\\& \geq \big(\vert \langle h_{g_w}^\ast, g_w\rangle\vert - \vert \langle h_{g_w}^\ast, \Vert w\Vert^{-1}w-g_w\rangle\vert\big) \Vert w\Vert \\&\geq (1-\delta_{k+1})\Vert w\Vert ,
\end{align*}
so that (IV) holds for $m=k+1$ as desired.

We now define $p_{k+1}$ in such a way that (V) holds. To this end define $v_{k+1}:= \Vert f_{\tau(k+1)}^\ast - f_{\tau(k+1)^-}^\ast\Vert^{-1}(f_{\tau(k+1)}^\ast-f_{\tau(k+1)^-}^\ast)$ and note that by Helly's theorem, by the density of $(z_p)_{p=0}^\infty$ in $S_Z$, and by the total boundedness of $S_{[(v_m)_{m=1}^{k+1}]}$ and $S_{[(v_m)_{m=1}^{k+1}]^\ast}$, we may choose $p_{k+1}>p_k$ large enough that (V) holds for $m=k+1$.

To complete the induction we now show that (VI) holds for $m=k+1$. Since (III) holds for $1\leq m\leq k+1$ and since $z_2^\ast \in \yoo_3$, it follows from the definition of $f_{\tau(k+1)}^\ast$ that
\[
\Vert f_{\tau(k+1)}^\ast -u^\ast\Vert \leq\frac{1}{2} \frac{\delta_{k+1}}{3}\frac{(1-3\nu)\epsilon}{(2+\nu)\Vert R\Vert}.
\]
Moreover, since $z_2^\ast \in \yoo_4$ and $u^\ast-f_{\tau(k+1)^-}^\ast=z_2^\ast-z_1^\ast$ we deduce that, for all $p\leq p_k$,
\begin{align}\label{justcaught}
\vert \langle f_{\tau(k+1)}^\ast -f_{\tau(k+1)^-}^\ast ,z_p\rangle\vert &\leq \Vert f_{\tau(k+1)}^\ast -u^\ast\Vert \Vert  z_p\Vert+\vert \langle u^\ast -f_{\tau(k+1)^-}^\ast ,z_p\rangle\vert \notag \\&\leq \frac{1}{2} \frac{\delta_{k+1}}{3}\frac{(1-3\nu)\epsilon}{(2+\nu)\Vert R\Vert}+ \frac{1}{2} \frac{\delta_{k+1}}{3}\frac{(1-3\nu)\epsilon}{(2+\nu)\Vert R\Vert} \notag\\& =\frac{\delta_{k+1}}{3}\frac{(1-3\nu)\epsilon}{(2+\nu)\Vert R\Vert}.
\end{align}
Since \begin{align*} \Vert f_{\tau(k+1)}^\ast -f_{\tau(k+1)^-}^\ast \Vert &\geq \Vert R^\ast f_{\tau(k+1)}^\ast -R^\ast f_{\tau(k+1)^-}^\ast \Vert \Vert R\Vert^{-1}\\& \geq \langle f_{\tau(k+1)}^\ast -f_{\tau(k+1)^-}^\ast, Rw_{\tau(k+1)}\rangle \Vert R\Vert^{-1}\\&> \frac{(1-3\nu)\epsilon}{(2+\nu)\Vert R\Vert},\end{align*} it follows from (\ref{justcaught}) that $\vert \langle v_{k+1},z_p\rangle\vert \leq \delta_{k+1}/3$ for all $p\leq p_k$. Thus (I)-(VI) hold for all $m<\omega$. The proof of Theorem~\ref{scoopie} is complete.
\end{proof}

\section{Subspaces and quotients with a basis}\label{basissection}

The following theorem was first established independently by Hagler \cite{Hagler1987} and Rosenthal \cite{Rosenthal1988}, who in doing so answered in the affirmative the following question posed by V. Zizler: {\em Suppose $X$ is a separable Banach space with $X^\ast$ nonseparable. Does $X$ contain a basic sequence whose closed linear span has non-separable dual?} The proof we offer here of the Hagler-Rosenthal result is a rather straightforward consequence of Theorem~\ref{scoopie}.

\begin{theorem}\label{secondbasistheorem}
Let $X$ be a separable Banach space with $X^\ast$ non-separable. Then for every $\delta>0$ there exists a subspace $Y\subseteq X$ such that $Y^\ast$ is non-separable and $Y$ has a basis with basis constant not exceeding $1+\delta$.
\end{theorem}

\begin{proof}
Let $\Omega$ be the tree of finite sequences of finite ordinals, as defined in Example~\ref{fullcounter}. Fix $\delta>0$ and $\tau:\omega \longrightarrow \Omega$ a bijection such that $\tau(m)\sqsubseteq\tau(m')$ implies $m\leq m'$. We apply Theorem~\ref{scoopie} with $W=Z=X$, $R=Id_X$, $K=B_{X^\ast}$, $\theta=1$, $\tee=\Omega$, and $\epsilon>0$ small enough that $B_{X^\ast}$ contains an uncountable $\epsilon$-separated subset, to obtain families $(z_t^\ast)_{t\in\overline{\Omega}}\subseteq B_{X^\ast}$ and $(w_t)_{t\in\Omega}\subseteq S_X$ such that 
\begin{equation*}
\langle z_t^\ast ,w_s\rangle=\begin{cases}
\langle z_s^\ast,w_s\rangle>\frac{\epsilon}{3}&\mbox{if }s\sqsubseteq t\\ 0 &\mbox{if }s\not\sqsubseteq t
\end{cases}, \quad s\in\Omega,\,t\in\overline{\Omega}\,.
\end{equation*}
and $(w_{\tau(m)})_{m=0}^\infty$ is a basic sequence with basis constant not exceeding $1+\delta$. Let $Y=[(w_t)_{t\in\Omega}]$ and for each $t\in\overline{\Omega}$ let $y_t^\ast= z_t^\ast|_Y$. For $t,t'\in\partial\Omega$ with $t\neq t'$ we have
\[
\Vert y_t^\ast - y_{t'}^\ast\Vert \geq \langle y_t^\ast - y_{t'}^\ast, w_{\min (t\setminus t')}\rangle = \langle z_t^\ast , w_{\min (t\setminus t')}\rangle - \langle z_{ t'}^\ast , w_{\min (t\setminus t')}\rangle >\frac{\epsilon}{3}-0 = \frac{\epsilon}{3}\,,
\]
hence $\{y_t^\ast\mid t\in\partial\Omega \}$ is an uncountable $\frac{\epsilon}{3}$-separated subset of $Y^\ast$.
\end{proof}

We now turn our attention to quotients of separable Banach spaces with non-separable dual. As far as we are aware, the following result is new.

\begin{theorem}\label{yetanotherbasistheorem}
Let $X$ be a separable Banach space with $X^\ast$ non-separable. Then for every $\delta>0$ there exists a subspace $Y\subseteq X$ such that $(X/Y)^\ast$ is non-separable and $X/Y$ has a basis with basis constant not exceeding $1+\delta$.
\end{theorem}

\begin{proof}
Let $\Omega$ be the tree of finite sequences of finite ordinals, as defined in Example~\ref{fullcounter}. Fix $\delta>0$ and $\tau:\omega \longrightarrow \Omega$ a bijection such that $\tau(m)\sqsubseteq\tau(m')$ implies $m\leq m'$. Let $\epsilon>0$ be small enough that $B_{X^\ast}$ contains an uncountable, $\epsilon$-separated subset. An application of Theorem~\ref{scoopie} with $\theta=1$, $W=Z=X$, $R=Id_X$, $K=B_{X^\ast}$ and $\tee=\Omega$ yields families $(w_t)_{t\in\Omega}\subseteq S_X$ and $(z_t^\ast)_{t\in\overline{\Omega}}\subseteq B_{X^\ast}$ such that
\begin{equation*}
\langle z_t^\ast ,w_s\rangle=\begin{cases}
\langle z_s^\ast,w_s\rangle>\frac{\epsilon}{3}&\mbox{if }s\sqsubseteq t\\ 0 &\mbox{if }s\not\sqsubseteq t
\end{cases}, \quad s\in\Omega,\,t\in\overline{\Omega}\,.
\end{equation*}
and the map $\Xi: t\mapsto z_t^\ast$ from $\overline{\Omega}$ to $X^\ast$ is coarse-wedge-to-weak$^\ast$ continuous. Let also $(f_t^\ast)_{t\in\Omega}$ and $(v_m)_{m=1}^\infty$ be as in the proof of Theorem~\ref{scoopie} and for $m\in[1,\omega)$ let \[ v_m^{\ast\ast}:= (\imath_{[(v_q)_{q=1}^\infty]}v_m)|_{[(v_q^\ast)_{q=1}^\infty]}. \]
Let $Y:= \bigcap_{t\in\Omega}\ker(z_t^\ast)$, let $Q:X\longrightarrow X/Y$ be the quotient map, and let $\overline{T}: X/Y\longrightarrow [(v_m^\ast)_{m=1}^\infty]$ be as in the proof of Theorem~\ref{scoopie}. As per Theorem~\ref{scoopie}(iv), we may assume $(Qw_{\tau(m)})_{m=0}^\infty$ is a basis for $X/Y$ with basis constant not exceeding $1+\delta$.
Moreover for $m<\omega$ and $x\in X$ we have
\begin{equation}\label{aldishopping}
\langle Q^\ast \overline{T}^\ast v_m^{\ast\ast},x\rangle = \langle v_m^{\ast\ast},\overline{T}Qx\rangle  = \langle Tx,v_m\rangle=\langle v_m,x\rangle,
\end{equation}
hence
\begin{equation}\label{couldntbe}
\spn\{ z_t^\ast \mid t\in\Omega\} = \spn\{ f_t^\ast\mid t\in\Omega\} = \spn\{ v_m\mid 1\leq m<\omega\}\subseteq Q^\ast \big( (X/Y)^\ast\big),
\end{equation}
where the first equality is immediate from the definitions, the second equality follows from the inductively verified fact that
\[
\forall\,k<\omega \quad \spn\{ f_{\tau(m)}^\ast\mid 1\leq m\leq k\}=\spn\{ v_m\mid 1\leq m\leq k\},
\] and the final inclusion follows from (\ref{aldishopping}). Since $\Omega$ is dense in $\overline{\Omega}$ in the coarse wedge topology, and since $Q^\ast((X/Y)^\ast)$ is weak$^\ast$-closed in $X^\ast$, it follows from the aforementioned continuity of $\Xi$ that $\{ z_t^\ast\mid t\in\overline{\Omega}\}\subseteq Q^\ast ((X/Y)^\ast)$. So for each $t\in\overline{\Omega}$ there exists (a unique) $y_t^\ast\in (X/Y)^\ast$ such that $Q^\ast y_t^\ast=z_t^\ast$, and then for $t,t'\in \partial\Omega$ with $t\neq t'$ we have
\[
\Vert y_t^\ast - y_{t'}^\ast\Vert \geq \langle y_t^\ast - y_{t'}^\ast, Qw_{\min (t\setminus t')}\rangle = \langle z_t^\ast , w_{\min (t\setminus t')}\rangle - \langle z_{t'}^\ast , w_{\min (t\setminus t')}\rangle >\frac{\epsilon}{3}-0 = \frac{\epsilon}{3}\,,
\]
hence $\{y_t^\ast\mid t\in\partial\Omega \}$ is an uncountable $\frac{\epsilon}{3}$-separated subset of $(X/Y)^\ast$.
\end{proof}

\section{Universal non-Asplund operators}\label{universalsection}
In this section we study the existence of universal elements of certain classes of non-Asplund operators. In order to state our main result we need to first define two types of operators associated to trees. After each definition we state a characterisation of the operators that factor the operator associated to a particular tree.

The following definition introduces the class of operators from which we shall draw our examples of universal non-Asplund operators.
\begin{definition}\label{discop}
Let $(\tee,\preceq)$ be a tree. Define $\Sigma_\tee\colon\ell_1(\tee)\longrightarrow \ell_\infty(\tee)$ by setting
\[
\Sigma_\tee w = \Big(\sum_{s\preceq t}w(s)\Big)_{t\in\tee}\,,\qquad w\in \ell_1(\tee).
\]
Equivalently, $\Sigma_\tee$ is the unique element of $\allop(\ell_1(\tee),\ell_\infty(\tee))$ satisfying $\Sigma_\tee e_t = \chi_{\tee[t\preceq]}$ for each $t\in\tee$.
\end{definition}
We have $\Vert \Sigma_\tee\Vert=1$ for every nonempty tree $\tee$. It is shown in \cite{PAHBSzlenkLarge} that every operator of the form $\Sigma_\tee$ is strictly singular.

We shall use the following proposition from \cite{PAHBSzlenkLarge} to determine whether $\Sigma_\tee$ factors through $T$, for certain trees $(\tee,\preceq)$ and operators $T$.

\begin{proposition}\label{factorkar}
Let $(\tee,\preceq)$ be a tree, $X$ and $Y$ Banach spaces and $T\in\allop(X,Y)$. The following are equivalent:
\begin{itemize}
\item[(i)] $\Sigma_\tee$ factors through $T$.
\item[(ii)] There exist $\delta>0$ and families $(x_t)_{t\in\tee}\subseteq B_X$ and $(x_t^\ast)_{t\in\tee}\subseteq T^\ast B_{Y^\ast}$ such that
\begin{equation}\label{deltadelta}
\langle x_t^\ast,x_s\rangle = \begin{cases}
\langle x_s^\ast,x_s\rangle\geq\delta&\mbox{if }s\preceq t\\
0&\mbox{if }s\npreceq t
\end{cases},
\quad\quad s,t\in\tee.
\end{equation}
\end{itemize}
\end{proposition}

The following definition introduces a class of operators that provide `continuous' analogues of $\Sigma_\tee$ for certain trees $\tee$. We shall draw from this class our examples of operators that are universal for the class of non-Asplund operators having separable codomain.

\begin{definition}
Let $\tee$ be a downwards-closed subtree of $\Omega$. Define $\sigma_\tee: \ell_1(\tee)\longrightarrow C(\overline{\tee})$ by
\begin{equation*}
(\sigma_\tee x)(t)=\begin{cases}
\sum_{s\sqsubseteq' t}x(s),& t\in\tee\\
\sum_{s\sqsubset' t}x(s),& t\in\partial\tee
\end{cases}, \quad x\in\ell_1(\tee), \,t\in\overline{\tee}.
\end{equation*}
That is, $\sigma_\tee$ is the unique element of $\allop(\ell_1(\tee),C(\overline{\tee}))$ that maps each $e_t\in\ell_1(\tee)$ to $\chi_{\tee[t\sqsubseteq']}\in C(\overline{\tee})$.
\end{definition}

The following `continuous' analogue of Proposition~\ref{factorkar} is also proved in \cite{PAHBSzlenkLarge}.

\begin{proposition}\label{kaktorfar}
Let $K$ be a compact Hausdorff space, $I$ an index set and $\{ K_i\}_{i\in I}$ a family of clopen subsets of $K$. For Banach spaces $X$ and $Y$ and $T\in\allop(X,Y)$ the following are equivalent:
\begin{itemize}
\item[(i)] $T$ factors the unique element of $\allop(\ell_1(I), C(K))$ satisfying $e_i\mapsto \chi_{K_i}$, $i\in I$.
\item[(ii)] There exists $(\delta_i)_{i\in I}\subseteq \real$ with $\inf_{i\in I}\delta_i>0$, a family $(x_i)_{i\in I}\subseteq X$ with $\sup_{i\in I}\Vert x_i\Vert <\infty$ and a weak$^\ast$-continuous map $\Xi: K\longrightarrow Y^\ast$ such that
\[
\forall\, i\in I\quad \forall\, k\in K \quad \langle \Xi (k), Tx_i\rangle = \begin{cases}
\delta_i, & k\in K_i\\
0, &k\notin K_i
\end{cases}.
\]
\end{itemize}
\end{proposition}

The following theorem is our main universality result for non-Asplund operators; some applications shall follow.

\begin{theorem}\label{puniuni}
Let $X$ and $Y$ be Banach spaces, $T\in\allop(X,Y)$ an operator such that $T^\ast(Y^\ast)$ is non-separable, and $\tee$ a downwards-closed subtree of $\Omega$. The following two statements hold:
\begin{itemize}
\item[(i)] If $X$ is separable then $\Sigma_\tee$ factors through $T$; and,
\item[(ii)] If $Y$ is separable then $\sigma_\tee$ factors through $T$, hence $\Sigma_\tee$ factors through $T$.
\end{itemize}
It follows that if $\ess$ is a downwards-closed subtree of $\Omega$ such that $\partial\ess$ is uncountable then:
\begin{itemize}
\item[(iii)] $\Sigma_\ess$ is universal for the class of non-Asplund operators; and,
\item[(iv)] $\sigma_\ess$ is universal for the class of operators in $\complement\sepop^\ast$ having separable codomain, which coincides with the class of non-Asplund operators having separable codomain.
\end{itemize}
In particular, (iii) and (iv) both hold in each of the following cases: 
\begin{itemize}
\item[(v)] $\ess=\Omega$, the tree of finite sequences of finite ordinals (c.f. Example~\ref{fullcounter}); and,
\item[(vi)] $\ess=\dee$, the infinite dyadic tree consisting of all finite sequences in $\{ 0,1\}$.
\end{itemize}
\end{theorem}

\begin{proof}
We begin by proving (i). Suppose $X$ is separable. An application of Theorem~\ref{scoopie} with $W=Z=X$, $R=Id_X$ $K=T^\ast (B_{Y^\ast})$, $\theta=1$, and $\epsilon>0$ small enough that $T^\ast B_{Y^\ast}$ contains an uncountable $\epsilon$-separated subset yields families $(x_t)_{t\in\tee}\subseteq S_X$ and $(x_t^\ast)_{t\in\tee}\subseteq T^\ast (B_{Y^\ast})$ such that:
\begin{equation}\label{yioryogger}
\langle x_t^\ast ,x_s\rangle=\begin{cases}
\langle x_s^\ast,x_s\rangle>\frac{\epsilon}{3}&\mbox{if }s\sqsubseteq t\\ 0 &\mbox{if }s\not\sqsubseteq t
\end{cases}, \qquad s,t\in\tee.
\end{equation}
An appeal to Proposition~\ref{factorkar} now establishes the validity of (i).

Next we prove (ii). Suppose $Y$ is norm separable. An application of Theorem~\ref{scoopie} with $W=X$, $Z=Y$, $R=T$, $K=B_{Y^\ast}$, $\theta=1$, and $\epsilon>0$ small enough that $T^\ast (B_{Y^\ast})$ contains an uncountable, $\epsilon$-separated subset yields families $(x_t)_{t\in\tee}\subseteq S_X$ and $(y_t^\ast)_{t\in\overline{\tee}}\subseteq B_{Y^\ast}$ such that
\begin{equation}\label{ziorzogger}
\langle y_t^\ast ,Tx_s\rangle=\begin{cases}
\langle y_s^\ast,Tx_s\rangle>\frac{\epsilon}{3}&\mbox{if }s\sqsubseteq' t\\ 0 &\mbox{if }s\not\sqsubseteq' t
\end{cases}, \qquad s\in\tee,\, t\in\overline{\tee},
\end{equation}
and the map $t\mapsto y_t^\ast$ is coarse-wedge-to-weak$^\ast$ continuous. It follows that $T$ satisfies condition (ii) of Proposition~\ref{kaktorfar} with $K=\overline{\tee}$, index set $I=\tee$, $K_t=\overline{\tee}[t\sqsubseteq']$ for each $t\in\tee$, and $\delta_t=\epsilon/3$ for all $t\in\tee$. We thus deduce from Proposition~\ref{kaktorfar} that $\sigma_\tee$ factors through $T$, which proves the first assertion of (ii). To see that the second assertion of (ii) holds, simply note that we have $\Sigma_\tee=E\sigma_\tee$, where $E$ is the isometric linear embedding $f\mapsto f|_\tee$ of $C(\overline{\tee})$ into $\ell_\infty(\tee)$.

To prove (iii) and (iv), suppose $(\ess,\sqsubseteq)$ is a downwards-closed subtree of $\Omega$ such that $\partial\ess$ is uncountable. We shall first show that $\Sigma_\ess$ and $\sigma_\ess$ are non-Asplund. Since the domain of $\sigma_\ess$, namely $\ell_1(\ess)$, is separable, the non-Asplundness of $\sigma_\ess$ will follow once we show that $\sigma_\ess^\ast$ has non-separable range. The non-Asplundness of $\Sigma_\ess$ will then follows since $\Sigma_\ess=E\sigma_\ess$ and $E^\ast$ is surjective, where $E$ is as defined in the preceding paragraph. For each $s\in\overline{\ess}$ let $\phi_s$ denote the evaluation functional of $C(\overline{\ess})$ at $s$ and note that for $s,s'\in\partial\ess$ with $s\neq s'$ we have \begin{align*}
\Vert \sigma_\ess^\ast\phi_{s} - \sigma_\ess^\ast\phi_{s'}\Vert &\geq \langle \sigma_\ess^\ast\phi_{s} - \sigma_\ess^\ast\phi_{s'}, e_{\min (s\setminus s')}\rangle \\& = \langle \phi_{s} , \chi_{\ess[\min (s\setminus s')\sqsubseteq']}\rangle - \langle \phi_{s'} , \chi_{\ess[\min (s\setminus s')\sqsubseteq']}\rangle \\&= 1-0 = 1\,,
\end{align*}
hence $\{\sigma_\ess^\ast\phi_{s} \mid s\in\partial\ess\}$ is an uncountable $1/2$-separated subset of $\ell_1(\ess)^\ast$. In particular, we have established that $\Sigma_\ess$ and $\sigma_\ess$ are non-Asplund. 

To see that $\Sigma_\ess$ factors through any non-Asplund operator, and is therefore universal for the class of non-Asplund operators - proving (iii) - let $W$ and $Y$ be Banach spaces and $S\in\allop(W,Y)\setminus\asplundop(W,Y)$. Then there exists a separable subspace $X$ of $W$ such that $S|_X\notin \sepop^\ast$. Letting $i:X\longrightarrow W$ denote the formal inclusion mapping, an application of (i) with $T=Si$ yields $U\in\allop(\ell_1(\ess),X)$ and $V\in\allop(Y,\ell_\infty(\ess))$ such that $\Sigma_\ess= V(Si)U=VS(iU)$. In particular, $\Sigma_\ess$ factors through $S$, so the universality of $\Sigma_\ess$ for $\complement\asplundop$ is proved.

We now verify that (iv) holds. Note that since $\sigma_\ess$ has separable codomain by definition and, as we have already shown, $\sigma_\ess\notin\sepop^\ast$, it follows from the first assertion of (ii) that $\sigma_\ess$ is universal for the class of operators in $\complement\sepop^\ast$ having separable codomain. If $Q$ is a non-Asplund operator with separable codomain then, by (iii), $\Sigma_\ess$ factors through $Q$, hence $Q\in\complement\sepop^\ast$ since $\Sigma_\ess\in\complement\sepop^\ast$. On the other hand, if $P\in\complement\sepop^\ast$ has separable codomain then, by (ii), $\sigma_\ess$ factors through $P$, hence $P$ is non-Asplund since $\sigma_\ess$ is non-Asplund. We have now proved (iv).

Finally, (v) and (vi) follow from (iii) and (iv) and the fact that, for either choice of $\ess$, $\ess$ is a downwards-closed subtree of $\Omega$ such that $\partial\ess$ is uncountable. 
\end{proof}

As the space $C(\{0,1\}^\omega)$ is perhaps a more familiar Banach space than $C(\overline{\dee})$, we note the following corollary of Theorem~\ref{puniuni}. Note that for $s\in\{0,1\}^\omega$ and $n<\omega$, $s|_n$ denotes the initial segment of $s$ of length $n$, so that $s|_n\in\{0,1\}^n$.

\begin{corollary}\label{otherunivsep}
Let $S:\ell_1(\dee)\longrightarrow C(\{0,1\}^\omega)$ be the operator satisfying
\[
(Sx)(s)=\sum_{n<\omega}x(s|_n)
\]
for every $x\in \ell_1(\dee)$ and $s\in \{0,1\}^\omega$ (that is, $Se_t=\chi_{\Delta_t}$ for each $t\in\dee$). Then $S$ factors through any operator having separable codomain and adjoint with non-separable range. It follows that $S$ is universal for the class of operators having separable codomain and adjoint with non-separable range.
\end{corollary}
\begin{proof}
We naturally identify $\partial \dee$ with $\{0,1\}^\omega$ by associating to each $s\in \{0,1\}^\omega$ its set of initial segments, $\{ s|_n\mid n<\omega\}$, which belongs to $\partial\dee$. Under this identification we have a homeomorphism, which we denote $\pi$, from $\{0,1\}^\omega$ (with its product topology) onto $\partial\dee$ (with the subspace topology it inherits as a closed subset of $\overline{\dee}=\dee\cup\partial\dee$ with respect to the coarse wedge topology). Let $S_\pi$ denote the isometric isomorphism from $C(\partial\dee)$ onto $C(\{0,1\}^\omega)$ given by setting $S_\pi f=f\circ \pi$ for each $f\in C(\partial\dee)$. Let $S_0$ denote the restriction operator $g\mapsto g|_{\partial\dee}$ from $C(\overline{\dee})$ to $C(\partial\dee)$. A routine calculation shows that $S=S_\pi S_0\sigma_\dee$, which by Theorem~\ref{puniuni}(ii) (with $\tee=\dee$) yields the first assertion of the corollary.

Since the codomain of $S$, namely $C(\{0,1\}^\omega)$, is separable, the second assertion of the corollary will follow once we observe that $S^\ast$ has non-separable range. That this is true follows from a similar argument to that used to produce an uncountable $1/2$-separated set in the proof of Theorem~\ref{puniuni}(iv); we omit the details.
\end{proof}

The following result is a separable codomain version of Stegall's universal operator theorem (Theorem~\ref{thewox} above). Here, as previously in the current paper, $(e_m)_{m=1}^\infty$ denotes the standard unit vector basis of $\ell_1$ and $(h_m)_{m=0}^\infty$ denotes the Haar basis of $C(\{0,1\}^\omega)$.

\begin{corollary}\label{trumpet}
Let $X$ and $Y$ be Banach spaces such that $Y$ is separable and suppose $T\in\allop(X,Y)$ is such that $T^\ast(Y^\ast)$ is non-separable. Let $J$ denote the unique element of $\allop(\ell_1,C(\{0,1\}^\omega))$ satisfying $Je_m=h_{m-1}$ for all $m\in\nat$. Then $J$ factors through $T$, hence $J$ is universal for the class of operators having separable codomain and adjoint with non-separable range. 
\end{corollary}

\begin{proof}
Let $(t_m)_{m=1}^\infty$ be the enumeration of $\dee$ given in the paragraph preceding the statement of Theorem~\ref{thewox} above. By Corollary~\ref{otherunivsep} there exist $U\in\allop(\ell_1(\dee),X)$ and $V\in (Y,C(\{0,1\}^\omega))$ such that $VTUe_t=\chi_{\Delta_t}$ for every $t\in\dee$. Define $U_0\in\allop(\ell_1,\ell_1(\dee))$ by setting $U_0e_1=e_\emptyset$ and $U_0e_m=e_{{t_{m-1}}^\smallfrown0}-e_{{t_{m-1}}^\smallfrown1}$ for $m>1$. We have\[
VTUU_0e_1 = VTUe_\emptyset=\chi_{\{0,1\}^\omega}=h_0
\]
and, for $m>1$,
\[
VTUU_0e_m = VTUe_{{t_{m-1}}^\smallfrown0}-VTUe_{{t_{m-1}}^\smallfrown1}=\chi_{\Delta_{{t_{m-1}}^\smallfrown0}}-\chi_{\Delta_{{t_{m-1}}^\smallfrown1}}=h_{m-1}.
\]
It follows that for such $U_0$ we have $VTUU_0=J$. 

To complete the proof we now show that $J^\ast(C(\{0,1\}^\omega)^\ast)$ is non-separable. For each $s\in\{0,1\}^\omega$ let $\varphi_s$ denote the evaluation functional of $C(\{0,1\}^\omega)$ at $s$. Suppose $s_0$ and $s_1$ are distinct elements of $\{0,1\}^\omega$ and let $n'=\min\{ n<\omega \mid s_0(n)\neq s_1(n)\}$. Without loss of generality, assume $s_0(n')=0$ and $s_1(n')=1$ and let $m$ be the finite ordinal such that $s_0\vert_{n'+1}={t_m}^\smallfrown0$ and $s_1\vert_{n'+1}={t_m}^\smallfrown1$. Then $h_m(s_0)=1$ and $h_m(s_1)=-1$, hence
\[
\Vert J^\ast \varphi_{s_0}-J^\ast \varphi_{s_1}\Vert \geq \langle J^\ast \varphi_{s_0}-J^\ast \varphi_{s_1}, e_{m+1}\rangle =  \langle \varphi_{s_0}, h_m\rangle-\langle \varphi_{s_1}, h_m\rangle=1-(-1)=2.
\]
It follows that $\{ J^\ast\varphi_s\mid s\in\{0,1\}^\omega\}$ is an uncountable, 1-separated subset of $J^\ast(C(\{0,1\}^\omega)^\ast)$. In particular, $J^\ast(C(\{0,1\}^\omega)^\ast)$ is non-separable.
 \end{proof}
 
Note that the composition of the operator $J$ in Corollary~\ref{trumpet} with the formal inclusion operator from $C(\{0,1\}^\omega)$ into $L_\infty(\{0,1\},\mu)$ coincides with Stegall's operator $H$ from Theorem~\ref{thewox}; from this fact it is straightforward to deduce Theorem~\ref{thewox} using Corollary~\ref{trumpet} and the injectivity of $L_\infty(\{0,1\}^\omega,\mu)$.

\section{A characterisation of Banach spaces with non-separable dual}\label{pelchsection}

A well-known result of Pe{\l}czy{\'n}ski asserts that a separable Banach space $X$ contains a subspace isomorphic to $\ell_1$ if and only if there exists a continuous linear surjection of $X$ onto $C(\{0,1\}^\omega)$ (see Theorem~3.4 and the remark on p.242 of \cite{Pelczynski1968}). By the Open Mapping Theorem, this is equivalent to the existence of $T\in\allop(X,C(\{0,1\}^\omega))$ such that $B_{C(\{0,1\}^\omega)}\subseteq T(B_X)$. It was noted by James on p.521 of \cite{James1950} that $X$ has non-separable dual if it contains a subspace isomorphic to $\ell_1$. James later showed that the converse statement does not hold in general for separable $X$ by constructing the James tree space which is separable, contains no copy of $\ell_1$, and has nonseparable dual \cite{James1974}. Thus $X$ having non-separable dual is a strictly weaker property than $X$ having a subspace isomorphic to $\ell_1$. 

Against the background of Pe{\l}czy{\'n}ski's aforementioned characterisation of separable Banach spaces containing a copy of $\ell_1$, E. Bator \cite{Bator1992} asked whether separable Banach spaces having non-separable dual may also be characterised in terms of a property of an operator from $X$ into $C(\{0,1\}^\omega)$. An answer to Bator's question was provided by M. Lopez Pellicer \cite{LopezPellicer1999}, who noted that the following two statements are equivalent for a separable Banach space $X$:
\begin{itemize}
\item[(a)] $X^\ast$ is non-separable; and
\item[(b)] for every $\epsilon>0$ there exists $T\in\allop(X, C(\{0,1\}^\omega))$ such that $\Vert T\Vert \leq 1+\epsilon$ and $\{ \chi_{\Delta_t}\mid t\in\dee \}\subseteq T(B_X)+\epsilon B_{C(\{0,1\}^\omega)}$.
\end{itemize}
The key observation in establishing the equivalence of (a) and (b) above is to note the implication (a)$\Rightarrow$(b) follows from an application of Theorem~\ref{doodah} with $\epsilon$ chosen suitably small. The following result of the current paper offers a stronger answer to Bator's question.

\begin{theorem}\label{weakerbeaker}
Let $X$ be a separable Banach space. The following are equivalent:
\begin{itemize}
\item[(i)] $X^\ast$ is non-separable.
\item[(ii)] For every $\epsilon>0$ there exists $T\in\allop(X,C(\{0,1\}^\omega))$ such that $\Vert T\Vert \leq 1+\epsilon$ and $\{ \chi_{\Delta_t}\mid t\in\dee \}\subseteq T(B_X)$.
\item[(iii)] For every $\epsilon>0$ there exists $R\in\allop(X,C(\{0,1\}^\omega))$ such that $\Vert R\Vert \leq 2+\epsilon$ and $(h_m)_{m=0}^\infty\subseteq R(B_X)$.
\end{itemize}
\end{theorem}

The equivalence of statements (i) and (iii) in Theorem~\ref{weakerbeaker} may be obtained by choosing $\epsilon>0$ suitably small and applying Theorem~\ref{doodah} and the Paley-Wiener stability theorem (see, e.g., Theorem~1.3.9 of \cite{Albiac2006}). Moreover, we may further obtain the equivalence of (i), (ii) and (iii) from Theorem~\ref{scoopie} if we omit the estimates $\Vert T\Vert \leq 1+\epsilon$ and $\Vert R\Vert \leq 2+\epsilon$ from (ii) and (iii), respectively (note that in any case we have $\Vert T\Vert \geq 1$ and $\Vert R\Vert \geq 1$). To obtain the full strength of Theorem~\ref{weakerbeaker} as stated we use instead the following lemma; our notation is as per previous sections of the current paper.
\begin{lemma}\label{amalgam}
Let $X$ be a separable Banach space such that $X^\ast$ is non-separable, let $\tee$ be a downwards-closed subtree of $\Omega$, and let $\varrho\in(0,1)$. Then there exist families $(x_t)_{t\in\tee}\subseteq (1+\varrho)B_X$ and $(x_t^\ast)_{t\in\overline{\tee}}\subseteq (1+\varrho)B_{X^\ast}$ such that
\begin{equation}\label{poweroftwo}
\langle x_t^\ast, x_s\rangle = \begin{cases}
1,& s\sqsubseteq' t\\
0,& s\not\sqsubseteq' t
\end{cases}, \qquad s\in\tee,\,t\in\overline{\tee}
\end{equation}
and the map $t\mapsto x_t^\ast$ from $\overline{\tee}$ to $(1+\varrho)B_{X^\ast}$ is coarse-wedge-to-weak$^\ast$ continuous.
\end{lemma}

We defer the proof of Lemma~\ref{amalgam} until after the proof of Theorem~\ref{weakerbeaker}.

\begin{proof}[Proof of Theorem~\ref{weakerbeaker}]
We will show that (i)$\Rightarrow$(ii)$\Rightarrow$(iii)$\Rightarrow$(i).

Suppose that (i) holds. Fix $\epsilon>0$ and let $\varrho\in(0,1)$ be small enough that $(1+\varrho)^2\leq1+\epsilon$. An application of Lemma~\ref{amalgam} with $\tee=\dee$ yields families $(x_t)_{t\in\dee}\subseteq (1+\varrho)B_X$ and $(x_t^\ast)_{t\in\overline{\dee}}\subseteq (1+\varrho)B_{X^\ast}$ such that
\begin{equation}\label{towerofpoo}
\langle x_t^\ast, x_s\rangle = \begin{cases}
1,& s\sqsubseteq' t\\
0,& s\not\sqsubseteq' t
\end{cases} ,\quad s\in\dee,\,t\in\overline{\dee},
\end{equation}
and the map $t\mapsto x_t^\ast$ from $\overline{\dee}$ to $(1+\varrho)B_{X^\ast}$ is coarse-wedge-to-weak$^\ast$ continuous. Define $T_0:X\longrightarrow C(\overline{\dee})$ by setting \[
(T_0x)(t) = (1+\varrho)\langle x_t^\ast,x\rangle
\]
for each $x\in X$ and $t\in\overline{\dee}$. Let $S_0$ and $S_\pi$ be as in the proof of Corollary~\ref{otherunivsep}, so that $\Vert S_0\Vert =1=\Vert S_\pi\Vert$ and $\Vert T_0\Vert \leq (1+\varrho)^2\leq1+\epsilon$. Define $T=S_\pi S_0T_0$, noting that $\Vert T\Vert\leq \Vert T_0\Vert\leq1+\epsilon$. For each $t\in\dee$ we have $(1+\varrho)^{-1}x_t \in B_X$ and, by (\ref{towerofpoo}), \[T\big((1+\varrho)^{-1}x_t\big)=\chi_{\Delta_t}\in B_{C(\{0,1\}^\omega)},\] hence (ii) holds. That is, (i)$\Rightarrow$(ii).

We now show that (ii)$\Rightarrow$(iii). Suppose (ii) holds and fix $\epsilon>0$. Let $T\in\allop(X,C(\{0,1\}^\omega))$ be such that $\Vert T\Vert \leq 1+\frac{\epsilon}{2}$ and $\{ \chi_{\Delta_t}\mid t\in\dee \}\subseteq T(B_X)$. Since each element of the Haar basis is the difference of two elements of $\{ \chi_{\Delta_t}\mid t\in\dee \}$, we see that (iii) holds by taking $R=2T$.

Finally, suppose that (iii) holds and let $R\in\allop(X,C(\{0,1\}^\omega))$ be such that $(h_m)_{m=0}^\infty\subseteq R(B_X)$. For each $m<\omega$ let $x_m\in B_X$ be such that $Rx_m=h_m$. Let $U$ be the (unique) element of $\allop(\ell_1,X)$ satisfying $Ue_m=x_{m-1}$ for each $m\in\nat$. Then $RU=J$, where $J$ is the operator defined in the statement of Corollary~\ref{trumpet}. As $J^\ast$ has non-separable range (as shown in the proof of Corollary~\ref{trumpet}) it follows that $X^\ast$ is non-separable since $J^\ast$ factors through $X^\ast$.
\end{proof}

The proof of Lemma~\ref{amalgam} combines the perturbative techniques from the proof of Theorem~\ref{scoopie} with an application of the following lemma due to Stegall.

\begin{lemma}[{\cite[Lemma~1]{Stegall1975}}]\label{ohman}
Let $Y$ be a non-separable Banach space and $\epsilon>0$. Then there exist families $(y_\alpha)_{\alpha<\omega_1}\subseteq S_Y$ and $(y_\alpha^\ast)_{\alpha<\omega_1}\subseteq (1+\epsilon)B_{Y^\ast}$ such that
\[
\langle y_\beta^\ast,y_\alpha\rangle = \begin{cases}
1,&\alpha=\beta\\
0,&\alpha<\beta
\end{cases}, \quad \alpha,\beta<\omega_1.
\]
\end{lemma}

\begin{proof}[Proof of Lemma~\ref{amalgam}]
Fix $\varrho\in(0,1)$ and let $\eta=\varrho/2$. Let $\tau:\omega\longrightarrow\tee$ be a bijection such that $\tau(l)\sqsubseteq \tau(m)$ implies $l\leq m$. Let $(z_p)_{p=0}^\infty$ be a norm dense sequence in $S_X$ and let $d$ be the metric on $X^\ast$ given by setting $d(x^\ast,y^\ast)=\sum_{p=0}^\infty 2^{-p-1}\vert\langle x^\ast -y^\ast,z_p\rangle\vert$ for $x^\ast,y^\ast\in X^\ast$. Recall that $d(x^\ast,y^\ast)\leq\Vert x^\ast-y^\ast\Vert$ for all $x^\ast,y^\ast\in X^\ast$ and that on bounded subsets of $X^\ast$ the weak$^\ast$ topology coincides with the topology induced by $d$.

By Lemma~\ref{ohman} there exist families $(z_\alpha^\ast)_{\alpha<\omega_1}\subseteq (1+\eta)B_{X^\ast}$ and $(z_\alpha^{\ast\ast})_{\alpha<\omega_1}\subseteq B_{X^{\ast\ast}}$ such that
\begin{equation}\label{crinemaft}
\langle z_\beta^{\ast\ast},z_\alpha^\ast\rangle = \begin{cases}
1,&\alpha=\beta\\
0,&\alpha<\beta
\end{cases}, \quad \alpha,\beta<\omega_1.
\end{equation}
Since the weak$^\ast$ topology on $(1+\eta)B_{X^\ast}$ is separable and metrisable, we may assume (see, e.g., Section~III.3, Problem~4 of \cite{Dieudonne1960}) that $z_\beta^\ast$ is a weak$^\ast$-condensation point of $\{ z_\alpha^\ast\mid\alpha<\omega_1\}$ in $(1+\eta)B_{X^\ast}$ for every $\beta<\omega_1$. We shall construct by induction on $m<\omega$ a strictly increasing sequence $(\alpha_m)_{m=0}^\infty$ of countable ordinals and sequences $(x_{\tau(m)})_{m=0}^\infty \subseteq (1+\eta)B_X$ and $(x_{\tau(m)}^\ast)_{m=0}^\infty \subseteq (1+2\eta)B_{X^\ast}$ satisfying the following conditions for each $m<\omega$:
\begin{itemize}
\item[(A)] $x_{\tau(m)}^\ast \in z_{\alpha_m}^\ast +(\sum_{j=1}^m2^{-j})\eta B_{\spn\{ z_{\alpha_j}^\ast\mid 0\leq j< m\}}\subseteq (1+2\eta)B_{\spn\{ z_{\alpha_j}^\ast\mid 0\leq j\leq m\}}
$;
\item[(B)] If $m>0$ then $d(x_{\tau(m)}^\ast,x_{\tau(m)^-}^\ast)<2^{-m}$; and
\item[(C)] For all $i,j\leq m$,
 \begin{equation}
\langle x_{\tau(j)}^\ast ,x_{\tau(i)}\rangle=\begin{cases}
1,&\tau(i)\sqsubseteq \tau(j)\\ 0, &\tau(i)\not\sqsubseteq \tau(j)
\end{cases};
\end{equation}
\end{itemize}

Once the existence of the aforementioned sequences $(\alpha_m)_{m=0}^\infty$, $(x_{\tau(m)})_{m=0}^\infty $ and $(x_{\tau(m)}^\ast)_{m=0}^\infty $ has been established, the proof of Lemma~\ref{amalgam} proceeds as follows. For $b\in\partial\tee$ and $n<\omega$ let $b_n$ denote the element of $\tee[\sqsubset b_n]$ of order type $n$. Let $x_b^\ast$ denote the weak$^\ast$ limit of the sequence $(x_{b_n}^\ast)_{n=0}^\infty\subseteq (1+2\eta)B_{X^\ast}$, which is Cauchy with respect to $d$ since (B) holds for all $m<\omega$. For $t\in\tee$ and $b\in\partial\tee$ it follows from the fact that (C) holds for all $m<\omega$ that
\[
\langle x_b^\ast, x_t\rangle=\begin{cases}
1,&\mbox{if }t\in b\\
0,&\mbox{if }t\notin b
\end{cases}.
\]
In particular, for $(x_t)_{t\in\tee}\subseteq (1+\eta)B_X$ and $(x_t^\ast)_{t\in\overline{\tee}}\subseteq (1+2\eta)B_{X^\ast}$ as defined above, we have that (\ref{poweroftwo}) holds. Since (B) holds for all $m<\omega$, the argument used to establish the continuity of the map $\Xi$ in the proof of Theorem~\ref{scoopie} shows that the map $t\mapsto x_t^\ast$ from $\overline{\tee}$ to $(1+2\eta)B_{X^\ast}$ is coarse-wedge-to-weak$^\ast$ continuous..

We now complete the proof of the lemma by constructing $(\alpha_m)_{m=0}^\infty$, $(x_{\tau(m)})_{m=0}^\infty $ and $(x_{\tau(m)}^\ast)_{m=0}^\infty $. Let $\alpha_0=0$, $x_{\tau(0)}^\ast = z_0^\ast$ and choose $x_{\tau(0)} \in (1+\eta)B_X$ such that $\langle x_{\tau(0)}^\ast,x_{\tau(0)}\rangle=1$. Clearly conditions (A), (B) and (C) are satisfied for $m=0$.

Fix $k<\omega$ and suppose that $\alpha_m<\omega_1$, $x_{\tau(m)}\in (1+\eta)B_X$ and $x_{\tau(m)}^\ast\in (1+2\eta)B_{X^\ast}$ have been defined for $0\leq m\leq k$ in such a way that (A), (B) and (C) are satisfied for $0\leq m\leq k$. To carry out the inductive step we now show how to construct $\alpha_{k+1}\in (\alpha_k,\omega_1)$, $x_{\tau(k+1)}\in (1+\eta)B_X$ and $x_{\tau(k+1)}^\ast\in (1+2\eta)B_{X^\ast}$ in such a way that (A), (B) and (C) hold for $m=k+1$. To this end let $v^\ast\in(\sum_{j=1}^k2^{-j})\eta B_{\spn\{ z_{\alpha_j}^\ast\mid 0\leq j<k\}}$ be such that $x_{\tau(k+1)^-}^\ast = z_{\alpha_{k'}}^\ast+v^\ast$, where $k'$ denotes the unique finite ordinal such that $\tau(k')= \tau(k+1)^-$. Define
\begin{align*}
\veee_1&:= \bigcap_{i=0}^k\Big\{ x^\ast\in X^\ast \,\,\Big\vert\,\, \vert \langle x^\ast - z_{\alpha_{k'}}^\ast,x_{\tau(i)}\rangle\vert<\frac{\eta}{k2^{k+2}(1+2\eta)} \Big\};\mbox{ and,}\\
\veee_2&:=  \{ x^\ast \in X^\ast \mid d(x^\ast, z_{\alpha_{k'}}^\ast)< 2^{-k-2}\}.
\end{align*}
Since $z_{\alpha_{k'}}^\ast$ is a weak$^\ast$-condensation point of $\{ z_\alpha^\ast\mid \alpha<\omega_1\}$ in $(1+\eta)B_{X^\ast}$ there exists $\alpha_{k+1}>\alpha_k$ such that $z_{\alpha_{k+1}}^\ast\in\veee_1\cap\veee_2\cap(1+\eta)B_{X^\ast}$. 

For convenience we let $x_{\tau(0)^-}^\ast$ denote the zero element of $X^\ast$, notwithstanding the fact that $\tau(0)^-$ is undefined. Define \begin{align*} x_{\tau(k+1)}^\ast :&= z_{\alpha_{k+1}}^\ast +v^\ast-\sum_{l=0}^k \langle z_{\alpha_{k+1}}^\ast- z_{\alpha_{k'}}^\ast, x_{\tau(l)}\rangle(x_{\tau(l)}^\ast - x_{\tau(l)^-}^\ast)\\ &=z_{\alpha_{k+1}}^\ast +v^\ast-\sum_{l=0}^k \langle z_{\alpha_{k+1}}^\ast- x_{\tau(k+1)^-}^\ast+v^\ast, x_{\tau(l)}\rangle(x_{\tau(l)}^\ast - x_{\tau(l)^-}^\ast)
,\end{align*} so that
\begin{equation}\label{fanta}
x_{\tau(k+1)}^\ast\in z_{\alpha_{k+1}}^\ast+ \Big[\Big(\sum_{j=1}^k2^{-j} \Big)\eta+c \Big]B_{\spn\{ z_{\alpha_j}^\ast\mid 0\leq j\leq k\}}, 
\end{equation}
where $c$ is a scalar that, since $z_{\alpha_{k+1}}^\ast \in \veee_1$, may be taken to satisfy
\begin{align}
c\leq k\frac{\eta}{k2^{k+2}(1+2\eta)}2(1+2\eta) &=2^{-k-1}\eta\label{pepsi} \\&<2^{-k-2}\label{coke}
\end{align}
It follows from (\ref{fanta}) and the estimate of $c$ provided at (\ref{pepsi}) that (A) holds for $m=k+1$.

The estimate for $c$ provided at (\ref{coke}) implies that $\Vert x_{\tau(k+1)}^\ast -z_{\alpha_{k+1}}^\ast-v^\ast\Vert <2^{-k-2}$, hence, since $z_{\alpha_{k+1}}^\ast\in\veee_2$, we have
\begin{align*}
d(x_{\tau(k+1)}^\ast,x_{\tau(k+1)^-}^\ast)&\leq d(x_{\tau(k+1)}^\ast,z_{\alpha_{k+1}}^\ast +v^\ast)+ d(z_{\alpha_{k+1}}^\ast +v^\ast,x_{\tau(k+1)^-}^\ast)\\
&\leq \Vert x_{\tau(k+1)}^\ast-z_{\alpha_{k+1}}^\ast -v^\ast\Vert + d(z_{\alpha_{k+1}}^\ast ,x_{\tau(k+1)^-}^\ast-v^\ast)\\
&<2^{-k-2}+ d(z_{\alpha_{k+1}}^\ast ,z_{\alpha_{k'}}^\ast)\\
&<2^{-k-2}+2^{-k-2}\\
&=2^{-k-1}.
\end{align*}
In particular, (B) holds for $m=k+1$.

To complete the induction it remains to define $x_{\tau(k+1)}$ and then verify that (C) holds for $m=k+1$. To this end apply Helly's theorem to obtain $x_{\tau(k+1)}\in (1+\eta)B_X$ such that $\langle x_{\tau(j)}^\ast, x_{\tau(k+1)}\rangle = \langle z_{\alpha_{k+1}}^{\ast\ast}, x_{\tau(j)}^\ast\rangle$ for $0\leq j \leq k+1$. Since (A) is assumed to hold for $0\leq m\leq k$, it follows from (\ref{crinemaft}) that for $0\leq j\leq k$ we have $\langle x_{\tau(j)}^\ast, x_{\tau(k+1)}\rangle = \langle z_{\alpha_{k+1}}^{\ast\ast}, x_{\tau(j)}^\ast\rangle= 0$. Moreover, since \[x_{\tau(k+1)}^\ast-z_{\alpha_{k+1}}^\ast\in\spn\{ z_{\alpha_j}^\ast\mid 0\leq j\leq k\}\subseteq \ker(z_{\alpha_{k+1}}^{\ast\ast}),\] it follows from (\ref{crinemaft}) that
\[
\langle x_{\tau(k+1)}^\ast,x_{\tau(k+1)}\rangle= \langle z_{\alpha_{k+1}}^{\ast\ast},x_{\tau(k+1)}^\ast\rangle = \langle z_{\alpha_{k+1}}^{\ast\ast},z_{\alpha_{k+1}}^\ast\rangle+\langle z_{\alpha_{k+1}}^{\ast\ast},x_{\tau(k+1)}^\ast-z_{\alpha_{k+1}}^\ast\rangle = 1.
\]
Since (C) is assumed to hold for $0\leq m\leq k$, if $l,i\in [0,k]$ then\[
\langle x_{\tau(l)}^\ast - x_{\tau(l)^-}^\ast, x_{\tau(i)}\rangle = \begin{cases}
1,&i=l\\
0,&i\neq l
\end{cases}.
\]
It follows that if $0\leq i\leq k$ then
\begin{align*}
\langle x_{\tau(k+1)}^\ast, x_{\tau(i)}\rangle &= \langle z_{\alpha_{k+1}}^\ast,x_{\tau(i)}\rangle+\langle v^\ast, x_{\tau(i)}\rangle -\langle z_{\alpha_{k+1}}^\ast -x_{\tau(k+1)^-}^\ast +v^\ast,x_{\tau(i)}\rangle\\
&= \langle x_{\tau(k+1)^-}^\ast, x_{\tau(i)}\rangle\\
&=\begin{cases}
1,&\tau(i) \sqsubseteq \tau(k+1)\\
0,&\tau(i)\not\sqsubseteq \tau(k+1)
\end{cases}.
\end{align*}
We have now shown that (C) holds for $m=k+1$, so the induction is complete.
\end{proof}

\section*{Acknowledgements}
The author thanks Professor Gilles Godefroy for bringing the reference \cite{Finet1989} to his attention.

\bibliographystyle{acm}

\vspace{4mm}

\noindent philip.a.h.brooker@gmail.com

\end{document}